\documentclass[notitlepage,11pt,reqno]{amsart}
\usepackage[foot]{amsaddr}
\usepackage{amssymb,nicefrac,bm,upgreek,mathtools,verbatim}
\usepackage[final]{hyperref}
\usepackage[mathscr]{eucal}
\usepackage{dsfont}
\usepackage[normalem]{ulem}
\usepackage{amsopn,esint}
\usepackage[T1]{fontenc}
\usepackage{appendix}
\usepackage[utf8]{inputenc}
\usepackage{apptools}
\AtAppendix{\counterwithin{lemma}{section}} 

\usepackage[margin=.95in]{geometry}

\newcommand{\stkout}[1]{\ifmmode\text{\sout{\ensuremath{#1}}}\else\sout{#1}\fi}

\newtheorem{lemma}{Lemma}[section]
\newtheorem{theorem}{Theorem}[section]
\newtheorem{proposition}{Proposition}[section]
\newtheorem{corollary}{Corollary}[section]

\theoremstyle{definition}
\newtheorem{definition}{Definition}[section]
\newtheorem{assumption}{Assumption}[section]

\newtheorem{example}{Example}[section]

\theoremstyle{remark}
\newtheorem{remark}{Remark}[section]
\numberwithin{theorem}{section}
\numberwithin{equation}{section}

\hypersetup{
  colorlinks=true,
  citecolor=black,
  linkcolor=black,
  urlcolor = blue,
  anchorcolor = blue,
  frenchlinks=false,
  pdfborder={0 0 0},
  naturalnames=false,
  hypertexnames=false,
  breaklinks}
\usepackage[capitalize,nameinlink]{cleveref}

\usepackage[abbrev,msc-links,nobysame,citation-order]{amsrefs}

\crefname{section}{Section}{Sections}
\crefname{appsec}{Appendix}{Appendices}
\crefname{subsection}{Section}{Sections}
\crefname{condition}{Condition}{Conditions}
\crefname{hypothesis}{Hypothesis}{Conditions}
\crefname{assumption}{Assumption}{Assumptions}
\crefname{lemma}{Lemma}{Lemmas}
\crefname{fact}{Fact}{Facts}

\Crefname{figure}{Figure}{Figures}

\crefformat{equation}{\textup{#2(#1)#3}}
\crefrangeformat{equation}{\textup{#3(#1)#4--#5(#2)#6}}
\crefmultiformat{equation}{\textup{#2(#1)#3}}{ and \textup{#2(#1)#3}}
{, \textup{#2(#1)#3}}{, and \textup{#2(#1)#3}}
\crefrangemultiformat{equation}{\textup{#3(#1)#4--#5(#2)#6}}%
{ and \textup{#3(#1)#4--#5(#2)#6}}{, \textup{#3(#1)#4--#5(#2)#6}}%
{, and \textup{#3(#1)#4--#5(#2)#6}}

\Crefformat{equation}{#2Equation~\textup{(#1)}#3}
\Crefrangeformat{equation}{Equations~\textup{#3(#1)#4--#5(#2)#6}}
\Crefmultiformat{equation}{Equations~\textup{#2(#1)#3}}{ and \textup{#2(#1)#3}}
{, \textup{#2(#1)#3}}{, and \textup{#2(#1)#3}}
\Crefrangemultiformat{equation}{Equations~\textup{#3(#1)#4--#5(#2)#6}}%
{ and \textup{#3(#1)#4--#5(#2)#6}}{, \textup{#3(#1)#4--#5(#2)#6}}%
{, and \textup{#3(#1)#4--#5(#2)#6}}

\crefdefaultlabelformat{#2\textup{#1}#3}
\newcommand{\vertiii}[1]{{\left\vert\kern-0.25ex\left\vert\kern-0.25ex\left\vert #1
    \right\vert\kern-0.25ex\right\vert\kern-0.25ex\right\vert}}

\newcommand{\sB}{{\mathscr{B}}}  
\newcommand{\cC}{{C}}   

\newcommand{\cI}{{\mathcal{I}}}  

\newcommand{\sK}{{\mathscr{K}}}  
\newcommand{\sL}{{\mathscr{L}}}  %

\newcommand{\cP}{{\mathcal{P}}} 
\newcommand{\cL}{{\mathcal{L}}} 

\newcommand{\sD}{\mathscr{D}}

\newcommand{\RR}{\mathds{R}}

\newcommand{\Rd}{{\mathds{R}^{d}}}

\newcommand{\D}{\mathrm{d}}
\newcommand{\E}{\mathrm{e}}
\newcommand{\Ind}{\mathds{1}}   

\newcommand{\Id}{\bm{I}}
\newcommand{\df}{\coloneqq}

\DeclareMathOperator*{\trace}{Tr}
\DeclareMathOperator*{\dist}{dist}
\DeclareMathOperator*{\diam}{diam}

\newcommand{\grad}{\nabla}

\newcommand{\abs}[1]{\lvert#1\rvert}
\newcommand{\norm}[1]{\lVert#1\rVert}

\usepackage{color}
\definecolor{dmagenta}{rgb}{.4,.1,.5}
\definecolor{dblue}{rgb}{.0,.0,.5}
\definecolor{mblue}{rgb}{.0,.0,.7}
\definecolor{ddblue}{rgb}{.0,.0,.4}
\definecolor{dred}{rgb}{.7,.0,.0}
\definecolor{dgreen}{rgb}{.0,.5,.0}
\definecolor{Eeom}{rgb}{.0,.0,.5}

\allowdisplaybreaks

\newcommand{\ttl}{\Large Fine boundary regularity for fully nonlinear mixed local-nonlocal problems}
\begin{document}
\title[Boundary regularity]
{\ttl}

\author{Mitesh Modasiya and Abhrojyoti Sen}

\address{
Department of Mathematics, Indian Institute of Science Education and Research, Dr. Homi Bhabha Road,
Pune 411008, India. Email: mitesh.modasiya@students.iiserpune.ac.in; abhrojyoti.sen@acads.iiserpune.ac.in }

\begin{abstract}
We consider Dirichlet problems for fully nonlinear mixed local-nonlocal non-translation invariant operators. For a bounded $C^2$ domain $\Omega \subset \Rd,$ let $u\in C(\Rd)$ be a viscosity solution of such Dirichlet problem. We obtain global Lipschitz regularity and fine boundary regularity for $u$ by constructing appropriate sub and supersolutions coupled with a \textit{Harnack type} inequality. We apply these results to obtain H\"{o}lder regularity of $Du$ up to the boundary.
\end{abstract}

\maketitle
\tableofcontents

\section{Introduction and main results}
In this article, for a bounded $C^2$ domain $ \Omega \subset \Rd,$ we establish the boundary regularity of the solution $u$ to the operator inequalities
\begin{equation}\label{Eq-1}
\begin{split}
\cL u + C_0 |Du| & \geq -K \; \quad \text{in}\; \Omega,
\\
\cL u - C_0 |Du| & \leq K  \quad \quad\text{in}\; \Omega,
\\
u & =0 \; \quad \quad\text{in}\; \Omega^c,
\end{split}
\end{equation}
where $C_0, K \geq 0$ and $\cL$ is a fully nonlinear integro-differential operator of the form
\begin{align}\label{eq1.1}
    \cL u(x):=\cL[x,u]=\sup_{\theta \in \Theta} \inf_{\nu \in \Gamma}\left\{ \trace  a_{\theta \nu}(x)D^2u(x) + \cI_{\theta \nu}[x,u]\right\},
\end{align}
for some index sets $\Theta , \Gamma$.
The coefficient $a_{\theta \nu}: \Omega \to \RR^{d\times d}$ is a matrix valued function and  $\cI_{\theta \nu}$ is a nonlocal operator defined as
   \begin{align}\label{nonlocal part}
  \cI_{\theta \nu} u(x):=\cI_{\theta \nu}[x,u]=\int_{\Rd} (u(x+y)-u(x) - \Ind_{B_1}(y) D u(x) \cdot y  )N_{\theta \nu}(x,y) \; \D{y}.
   \end{align}
The above operator inequalities stated in \eqref{Eq-1} are motivated by Hamilton-Jacobi equations of the form
\begin{align*}
Iu(x):=\sup_{\theta \in \Theta}\inf_{\nu \in \Gamma}\left\{L_{\theta \nu} u(x) + f_{\theta \nu}(x)\right\}=0,
\end{align*}
where
\begin{align}\label{full operator}
    L_{\theta \nu}u(x)=\trace  a_{\theta \nu}(x)D^2u(x) + \cI_{\theta \nu}[x,u]+b_{\theta \nu}(x)\cdot Du(x),
\end{align}
$b_{\theta\nu}(\cdot)$ and $f_{\theta\nu}(\cdot)$ are bounded functions on $\Omega.$ These linear operators \eqref{full operator} are extended generator for a wide class of $d$-dimensional Feller processes
(more precisely, jump diffusions) and the nonlinear operator $Iu(\cdot)$ has its connection to the stochastic control problems and differential games (see \cites{B12, BK22} and the references therein). The first term in \eqref{full operator} represents the diffusion, the second term represents the jump part of a Feller process, while the third corresponds to the drift term. We refer to \cites{C12, BSW, A} and the references therein for more on the connections between the operators of the form \eqref{full operator} and stochastic differential equations.  
For a precise application of these type of operators in finance and biological models, we refer to \cites{DPLV22, DV21, RTbook} and the references therein.

We set the following assumptions on the coefficient $a_{\theta \nu}(\cdot)$ and the kernel $N_{\theta \nu}(x,y)$, throughout this article.
\begin{assumption}\label{Assmp-1}
\hspace{-2em}
\begin{itemize}
\item[(a)] $a_{\theta\nu} (\cdot)$ are uniformly continuous and bounded in $\overline{\Omega}$, uniformly in $\theta , \nu$ for $\theta \in \Theta, \nu \in \Gamma.$  
Furthermore, $a_{\theta\nu} (\cdot)$ satisfies the uniform ellipticity condition  $\lambda \Id  \leq a_{\theta  \nu} (\cdot) \leq \Lambda \Id$ for some $0 < \lambda \leq \Lambda$  where $\Id$ denotes the $d\times d$ identity matrix.
\item[(b)] For each $\theta \in \Theta, \nu \in \Gamma$, $N_{\theta \nu}:\Omega \times \Rd$ is a measurable function and  for some $\alpha \in (0,2)$ there exists a kernel $k$ that is measurable in $\Rd\setminus \{0\}$ such that for any $\theta \in \Theta, \nu \in \Gamma, x\in \Omega,$ we have 
$$
0 \leq N_{\theta \nu}(x,y) \leq k(y)
$$
and 
$$\int_{\Rd} ( 1 \wedge |y|^{\alpha}) k(y) \D{y} < +\infty,
$$
where we denote $p\wedge q:=\min \{p,q\}$ for $p, q \in \RR.$
\end{itemize}
\end{assumption}

Let us briefly comment on \cref{Assmp-1}. The uniform continuity of $a_{\theta \nu}(\cdot)$ is required for the stability of viscosity sub- or supersolutions under appropriate limits and useful in \cref{A1} which is a key step for proving interior $C^{1,\gamma}$ regularity (cf. \cref{TH2.1}). The \cref{Assmp-1}(b) includes a large class of kernels. We mention some examples below.
\begin{example} Consider the following kernels $N_{\theta \nu} (x,y):$
\begin{itemize}
    \item[(i)] $N_{\theta \nu}(x,y)=\frac{1}{|y|^{d+\sigma}}$ for $\sigma \in (0,2).$ Clearly we can take $k(y)=\frac{1}{|y|^{d+\sigma}}$ and $\int_{\Rd}(1\wedge |y|^{\alpha})k(y)\D{y}$ is finite for $\alpha \in (\sigma, 2).$
    \item[(ii)] $N_{\theta \nu}(x,y)=\sum_{i=1}^{\infty}\frac{a_i}{|y|^{d+\sigma_i}}$ for $\sigma_i\in (0,2)$, $\sigma_0=\sup_{i}\sigma_i<2$ and $\sum_{i=1}^{\infty} a_i=1.$ Similarly taking $N_{\theta\nu}(x,y)=k(y)$ we can see $\int_{\Rd}(1\wedge |y|^{\alpha})k(y) < +\infty$ for $\alpha \in [1+\sigma_0/2, 2)$.
    \item[(iii)] $N_{\theta \nu}(x,y)=\begin{cases}
        \frac{(1-\log|y|)^{\beta}}{|y|^{d+\sigma}} \,\,\,\, \text{for}\,\,\,\, 0<|y|\leq 1\\
        \frac{(1+\log|y|)^{-\beta}}{|y|^{d+\sigma}}\,\,\,\, \text{for}\,\,\,\, |y|\geq 1,
    \end{cases}$\\
    where $\sigma \in (0,2).$\\
    (a) For $2(2-\sigma)> \beta \geq 0,$ taking $N_{\theta \nu}(x,y)=k(y)$ we have $\int_{\Rd}(1\wedge|y|^{\alpha})k(y)\D{y}<+\infty$ for $\alpha \in [1+\frac{\sigma}{2}+\frac{\beta}{4},2).$\\
    (b) For $-\sigma <\beta < 0,$ taking $N_{\theta \nu}(x,y)=k(y)$ we have $\int_{\Rd}(1\wedge|y|^{\alpha})k(y)\D{y}<+\infty$ for $\alpha \in [1+\frac{\sigma}{2},2).$\\
    \underline{Proof of (a):} 
    \begin{align*}
       \int_{\Rd} (1\wedge |y|^{\alpha})k(y)\D{y}=\int_{|y| \leq 1} \frac{|y|^{\alpha}(1-\log|y|)^{\beta}}{|y|^{d+\sigma}}\D{y}+\int_{|y|>1}\frac{(1+\log|y|)^{-\beta}}{|y|^{d+\sigma}}\D{y}:=I_1+I_2.
    \end{align*}
  Using $(1-\log|y|) \leq \frac{1}{\sqrt{|y|}}+1$ and the convexity of $\xi(t)=t^p$ for $p\geq 1$ we get
  \[(1-\log|y|)^{\beta} \leq  C \left(\frac{1}{|y|^{\beta/2}}+1\right).\] Therefore
  \begin{align*}
      I_1 \leq \int_{|y| \leq 1} \frac{ C \D{y}}{|y|^{\beta/2+d+\sigma-\alpha}} + \int_{|y|\leq 1} \frac{C \D{y}}{|y|^{d+\sigma-\alpha}}   <+\infty \,\,\,\, \text{for}\,\,\,\, \alpha \in [1+\sigma/2 +\beta/4, 2),
  \end{align*}
  and 
  \begin{align*}
      I_2 \leq \int_{|y|>1}\frac{\D{y}}{|y|^{d+\sigma}} <+\infty.
  \end{align*}
  \underline{Proof of (b):} Since $\beta <0$ in this case, we have $(1-\log|y|)^{\beta}\leq 1$ and $I_1 <+\infty$ for $\alpha \in [1+\frac{\sigma}{2}, 2).$ To estimate $I_2,$ observe $(1+\log|y|)^{-\beta}\leq (1+|y|)^{-\beta}$ and 
  \begin{align*}
      I_2 \leq C\int_{|y|>1}\frac{(1+|y|^{-\beta})}{|y|^{d+\sigma}}\D{y} < +\infty
\,\,\,\, \text{since}\,\,\,\, \sigma >-\beta.  
\end{align*}
\item[(iv)] $N_{\theta \nu}(x,y)= \frac{\Psi(1/|y|^2)}{|y|^{d+\sigma(x,y)}},$ where $\sigma: \Rd \times \Rd \to \RR$ satisfying 
\[0< \sigma^{-}:=\inf_{(x,y)\in \Rd\times \Rd}\sigma(x,y) \leq \sup_{(x,y)\in \Rd\times \Rd}\sigma(x,y):=\sigma^{+}<2.\] and $\Psi$ is a Bernstein function (for several examples of such functions, see \cite{SSV}) vanishing at zero. Furthermore, $\Psi$ is non-decreasing, concave and satisfies a \textit{weak upper scaling property} i.e, there exists $\mu \geq 0$ and $c \in (0,1]$ such that 
\[ \Psi(\lambda x) \leq c\lambda^{\mu}\Psi(x)\,\,\,\, \text{for}\,\,\,\, x\geq s_0>0, \lambda\geq 1.\]
For $ \mu < 2(2-\sigma^+)$, we can take
\begin{align*}
    k(y)=\begin{cases}
        \frac{\Psi(1)}{|y|^{d+2\mu+\sigma^{+}}}, \,\,\,\, &\text{if}\,\,\,\, 0<|y|\leq 1,\\
        \frac{\Psi(1)}{|y|^{d+\sigma^{-}}}, \,\,\,\, &\text{if}\,\,\,\, |y|> 1
\end{cases}
\end{align*}
and $\int_{\Rd}(1\wedge |y|^{\alpha})k(y)\D{y} < +\infty$ for $\alpha \in [1+\mu+\sigma^{+}/2, 2).$ 
\end{itemize}
 \end{example}
The main purpose of this article is to establish a global Lipschitz regularity and boundary regularity of the solutions satisfying \eqref{Eq-1} under the \cref{Assmp-1}. On the topic of regularity theory for linear elliptic equations, \textcolor{black}{an} H\"{o}lder estimate plays a key role and it can be obtained by using Harnack inequality. The pioneering contributions are by DeGiorgi-Nash-Moser \cites{EDG, JM, JN} who proved $C^{\alpha}$ regularity for solutions to the second order elliptic equations in divergence form with measurable coefficients under the assumption of uniform ellipticity. For equations of non-divergence form, the corresponding regularity theory was established by Krylov and Safonov \cite{KS79}. We refer to \cite{CCbook} for a comprehensive overview of the regularity theory for fully nonlinear elliptic equations.  In \cite{K83}, Krylov studied the boundary regularity for local second order elliptic equations in non-divergence form with bounded measurable coefficients. He obtained the H\"{o}lder regularity of $\frac{u}{\delta}$ up to the boundary where $\delta$ denotes the distance function, i.e., $\delta(x)=\dist(x, \Omega^c)$. 

Turning our attention towards the \textcolor{black}{case of} nonlocal equations, first H\"{o}lder estimates and Harnack inequalities for $s$-harmonic functions are proved by Bass and Kassmann \cites{BK05, BK2005}, and by Bass and Levin \cite{BL02}, using a purely probabilistic approach. \textcolor{black}{However, in the nonlocal setting, the classical form of the Harnack inequality requires the $s$-harmonic function to be nonnegative in the whole of $\RR^d.$ If this global nonnegativity is not assumed, a counterexample was constructed by Kassmann \cite[Theorem 1.2]{K2007}. This phenomenon can also be seen from the important work of Dipierro, Savin and Valdinoci \cite{DSV17} (also see \cite{Kr19}), where they show that all functions are locally $s$-harmonic up to a small error. As a consequence, it is possible to construct $s$-harmonic functions that arbitrarily oscillate and can reach extremas at any chosen point in the ball, more specifically, see \cite[Difference 2.5]{AV19}.}  

In the realm of analytic setup, Silvestre \cite{S06} proved H\"{o}lder continuity of $u$ satisfying \eqref{nonlocal part} with some structural assumptions on the operator and kernel related to the assumptions of Bass and Kassmann. Analogous to the local case \cite{K83}, in the nonlocal setting, for a bounded domain $\Omega\subset \Rd$ with $C^{1,1}$ boundary the first result concerning boundary regularity of $u$ solving the Dirichlet problem for $(-\Delta)^s$ with bounded right hand side is obtained by Ros-Oton and Serra \cite{RS14} where they established a H\"{o}lder regularity of $u/\delta^s$ up to the boundary. This result is proved by using a method of Krylov (see \cite{Kaz}). The idea is to obtain a bound for $u$ with respect to a constant multiple of $\delta^s$ and this controls the oscillation of $u/\delta^s$ near the boundary $\partial \Omega.$ The H\"{o}lder regularity of $u/\delta^s,$ (i) for more general nonlocal linear operators with $C^{1, \alpha}$ domain is established in \cite{RS17}, (ii) for smooth domain with smooth right hand side is established in \cites{G14, G15}, (iii) for kernel with variable order see \cite{KKLL} and (iv) for Dirichlet problem for fractional $p$-Laplacian, see \cite{IMS}. 

In a seminal paper, Caffarelli and Silvestre \cite{CS09} studied the regularity theory for fully nonlinear integro-differential equations of the form : $\sup_{\theta \in \Theta}\inf_{\nu \in \Gamma} \cI[x,u]$ where $\cI [x,u]$ is given by \eqref{nonlocal part}. By obtaining a nonlocal ABP estimate, they established the  H\"{o}lder regularity and Harnack inequality when $N_{\theta \nu}(y)$ ($N_{\theta\nu}(y)$ denotes the $x$-independent form of $N_{\theta\nu}(x,y)$) is positive, symmetric and comparable  with the kernel of the fractional Laplacian. From a large amount of literature that extend the work of Caffarelli and Silvestre \cite{CS09}, we mention \cite{KKL16} where the authors considered integro-PDEs with regularly varying kernel, \cites{BM20, CLD12, KL12} where regularity results are obtained for symmetric and non-symmetric stable-like operators and \cite{KL20} for kernels with variable order. Also, a recent paper \cite{K22} studies H\"{o}lder regularity and a scale invariant Harnack inequality under some weak scaling condition on the kernel. Boundary regularity results for fully nonlinear integro-differential equations are obtained by Ros-Oton and Serra in \cite{RS16}. They considered a restricted class of kernels $\sL_*$ where $N_{\theta \nu}(x,y)$ is $x$-independent and of the following form
\[N_{\theta \nu}(y):=\frac{\mu(y/|y|)}{|y|^{d+2s}}\,\,\,\,\, \text{with}\,\,\,\, \mu\in L^{\infty}(S^{d-1}),\] satisfying $\mu(\theta)=\mu(-\theta)$ and $\lambda\leq \mu \leq \Lambda$ where $0<\lambda\leq \Lambda$ are the ellipticity constants. An interesting feature of $\sL_*$ is 
\[L(x_d)_{+}^s=0\,\,\,\, \text{in}\,\,\,\, \{x_d>0\} \,\,\,\, \text{for all}\,\,\,\, L\in \sL_*\] which is useful to construct barriers in their case. Note that our operators do not enjoy such property for having different orders. Furthermore, with \cref{Assmp-1} the nonlocal part \eqref{nonlocal part} is not scale invariant in our case, that is one may not find any $0\leq \beta \leq 2$ such that $\cI_{\theta\nu}[x, u(r\cdot)]=r^{\beta}\cI_{\theta \nu}[rx, u(\cdot)]$ for any $0<r<1.$

Recently, the mathematical study of mixed local-nonlocal integro-differential equations have been received a considerable attention, for instance, see \cites{BM21, BDVV, BDVV22, BVDV, GM02, Barles12}. The regularity results and Harnack inequality for mixed fractional $p$-Laplace equations are recently obtained in \cites{GK22, GL22}. The interior $C^{\alpha}$ regularity theory for HJBI-type integro-PDEs has been studied by Mou \cite{M19}. He obtained H\"{o}lder regularity for viscosity solutions under uniform ellipticity condition and a slightly weaker condition on kernels in compared to the \cref{Assmp-1} (b), that is, $\int_{\Rd}(1\wedge|y|^2)k(y)\D{y} < +\infty.$ More recently global Lipschitz regularity (compare it with Biagi et.\ al. \cite{BDVV22}) and fine boundary regularity have been obtained for linear mixed local-nonlocal operators in \cite{BMS22}. Since the nonlocal operator applied on the distance function becomes singular near the boundary for certain range of order of the kernel,  one of the main challenges was to construct appropriate sub and supersolutions and prove an oscillation lemma following \cite{RS14}. To do such analysis, along with several careful estimates, the authors borrowed a Harnack inequality from \cite{F09}. Note that for fully nonlinear mixed operators of the form \eqref{eq1.1} no such Harnack inequality is available in the literature.

 In this current contribution, we continue the study started in \cite{BMS22} to obtain the boundary regularity for fully nonlinear integro-differential problems of the form \eqref{Eq-1}. Below, we present our first result that is the Lipschitz regularity of $u$ satisfying \eqref{Eq-1} up to the boundary. Note that \eqref{LPZ} can be achieved under some weaker assumptions on the domain and kernel. For this result, we only assume $\partial \Omega$ to be $C^{1,1}$ and $\int_{\Rd}(1 \wedge |y|^2) k(y)\D{y}<+\infty.$
\begin{theorem}[Lipschitz regularity]\label{T1.1}
Let $\Omega $ be a bounded $C^{1,1}$ domain in $\Rd$ and $u$ be a continuous function which solves the operator inequalities \eqref{Eq-1}
in viscosity sense. Then $u$ is in $C^{0,1}(\Rd)$ and  there exists a constant 
$C$, depending only on $d, \Omega, \lambda, \Lambda, C_0,$
$\int_{\Rd}(1 \wedge |y|^2) k(y) \D{y}$, such that
\begin{equation}\label{LPZ}
\Vert u\Vert_{C^{0,1}(\Rd)} \leq C K.
\end{equation}
\end{theorem}
To prove \cref{T1.1}, the first step is to show that the distance function $\delta(x)=\dist(x, \Omega^c)$ can be used as a barrier to $u$ in $\Omega.$ Once this is done, we can complete the proof by considering different cases depending on the distance between any two points in $\Omega$ or their distance from $\partial \Omega$ and combining $|u|\leq C\delta$ with an interior $C^{1, \gamma}$-estimate for scaled operators (cf. \cref{TH2.1}).

Next we show the fine boundary regularity, that is the H\"{o}lder regularity of $u/\delta$ up to the boundary. 

\begin{theorem} [Boundary regularity]\label{T1.2}
Suppose that \cref{Assmp-1} holds. Let $\Omega$ be a bounded $C^2$ domain and $u$ be a viscosity
solution to the operator inequalities \eqref{Eq-1}.
 Then there exists 
$\upkappa\in (0,  \hat{\alpha})$ such that 
\begin{equation}\label{BMS01}
\norm{u/\delta}_{C^{\upkappa}(\overline{\Omega})}\leq C_1 K,
\end{equation}
for some constant $C_1$, where
$\upkappa, C_1$ depend on $d, \Omega,  C_0, \Lambda,\lambda,  \alpha$ and $ \int_{\Rd} (1 \wedge |y|^\alpha 
 ) k(y) \D{y}$. Here $\hat{\alpha}$ is given by
\begin{align*}
    \hat\alpha=\begin{cases}
    1\,\,\,\, &if \,\,\,\, \alpha \in (0,1]\\
    \frac{2-\alpha}{2}\,\,\,\, &if \,\,\,\, \alpha \in (1,2).
        \end{cases}
\end{align*}
\end{theorem}
To prove \cref{T1.2}, following \cite{RS14}, we first establish an oscillation lemma (cf. \cref{P2.1}). For this purpose, it is necessary to construct appropriate sub- and supersolutions carefully since $\cI_{\theta \nu} \delta$ becomes singular near the boundary $\partial \Omega$ when $\alpha\in (1,2).$ Then we shall use a ``Harnack type'' inequality (cf. \cref{T5.1}). This \textit{Harnack type} inequality is new and had to be developed specifically for our setting due to the unavailability of classical Harnack inequality in this context. Moreover, it is important to note that one must bypass the use of the comparison principle \cite[Theorem 5.1]{BM21} in our analysis, since the mentioned theorem applies only for translation invariant linear operators. For non-translation invariant operators, such comparison principle is unavailable, see \cref{rem2.1} for further details.

Now applying \eqref{BMS01}, we prove the H\"{o}lder regularity of $Du$ up to the boundary. 
\begin{theorem} [Gradient H\"{o}lder regularity]\label{T1.3}
Suppose that \cref{Assmp-1} holds and $\Omega$ be a bounded $C^2$ domain. Then for any viscosity solution $u$ to the operator inequalities \eqref{Eq-1} we have
\begin{align*}
    ||Du||_{C^{\eta}(\overline{\Omega})} \leq CK,
\end{align*}
for some $\eta \in (0,1)$ and $C,$ depending only on $d, \Omega,  C_0, \Lambda,\lambda,  \alpha$ and $ \int_{\Rd} (1 \wedge |y|^\alpha 
 ) k(y) \D{y}$.
\end{theorem}
The interior $C^{1, \eta}$-regularity for fully nonlinear integro-differential equations is studied in \cite{CS09} by introducing a new ellipticity class where the kernels are $C^1$ away from the origin. Kriventsov \cite{K13} extended this result without the additional assumption on kernels (sometimes referred as \textit{rough kernels}). Also see \cite{S15} for its parabolic version. For HJBI-type integro-PDEs, interior $C^{1, \eta}$-regularity is established by Mou and Zhang \cite{MZ21} and for mixed local nonlocal fractional $p$-Laplacian, see \cite{DM22}. The $C^{1,\eta}$-regularity up to the boundary for linear mixed local-nonlocal operators is recently obtained in \cite{BMS22}.

The rest of the article is organized as follows. In \cref{section 2}, we introduce the necessary preliminaries and collect all the auxiliary results which will be used throughout the article. In \cref{Section 3} we prove \cref{T1.1}. \cref{T1.2} is proved in \cref{Section 4}. In \cref{section 5} we prove \cref{T1.3}. Lastly, in \cref{appendix}, following an approximation and scaling argument, we give a proof of $C^{1, \gamma}$ regularity for a scaled operator, i.e., \cref{TH2.1}.
\section{Preliminaries}\label{section 2}
This section sets the notation which we use throughout the paper and collects the necessary results.
\subsection{Notations and Definitions}
We start by setting the notation. We use $\sB_r(x)$ to denote an open ball of radius $r>0$ centered at a point $x \in \Rd$ and for $x=0,$ we denote $\sB_r:=\sB_r(0).$ For any subset $U\subseteq \Rd$ and for $\alpha \in (0,1),$ we denote $C^{\alpha}(U)$ as the space of all bounded, $\alpha$-H\"{o}lder continuous functions equipped with the norm
\[||f||_{C^{\alpha}(U)}:=\sup_{x\in U}|f(x)|+\sup_{x,y\in U}\frac{|f(x)-f(y)|}{|x-y|^{\alpha}}.\]
Recall that for $\alpha=1,$ $C^{0,1}(U)$ denotes the space of all Lipschitz continuous functions on $U.$ The space of all bounded functions with bounded $\alpha$-H\"{o}lder continuous derivatives is denoted by $C^{1,\alpha}(U)$ with the norm
\[||f||_{C^{1,\alpha}(U)}:=\sup_{x\in U}|f(x)|+ ||Df||_{C^{\alpha}(U)}.\]

We use $USC(\Rd), LSC(\Rd), C(\Rd), C_b(\Rd), M^d$ to denote the space of upper semicontinuous, lower semicontinuous, continuous functions, bounded continuous functions on $\Rd$ and $d\times d$ symmetric matrices respectively.

Now, let us introduce the scaled operators.
For $0<s \leq 1,$ we define scaled version of \eqref{eq1.1} as following.
\begin{align*}
    \cL^s[x,u]=\sup_{\theta \in \Theta} \inf_{\nu \in \Gamma}\left\{\trace a_{\theta \nu}(sx)D^2u(x) + \cI^s_{\theta \nu}[x,u]\right\},
\end{align*}
where 
\begin{align*}
    \cI^s_{\theta \nu}[x,u]=\int_{\mathbb{R}^d}(u(x+y) -u(x) - \Ind_{\sB_{\frac{1}{s}}}(y) \grad u(x) \cdot y   )s^{d+2}N_{\theta \nu}(sx, sy) \D{y}.
\end{align*}
Next we define extremal Pucci operators for second order term and the nonlocal term.  
\begin{align*}
\cP^{+}u(x) &\df \sup \left\lbrace \trace ( A D^2 u(x)) , A \in  M^d ,  \lambda \Id \leq A \leq \Lambda \Id  \right\rbrace, \\
\cP^{-}u(x) &\df \inf \left\lbrace \trace ( A D^2 u(x)) , A \in  M^d ,  \lambda \Id \leq A \leq \Lambda \Id  \right\rbrace, \\
\end{align*}
and 
\begin{align*}
\cP^{+}_{k,s} u(x) &\df \int_{\Rd} (u(x+y) -u(x) - \Ind_{\sB_{\frac{1}{s}}}(y) \grad u(x) \cdot y   )^{+} s^{d+2} k(sy) \D{y} , \\
\cP^{-}_{k,s} u(x) &\df - \int_{\Rd} (u(x+y) -u(x) - \Ind_{\sB_{\frac{1}{s}}}(y) \grad u(x) \cdot y   )^{-} s^{d+2} k(sy) \D{y} .
\end{align*}
Denote $\cP^+_{k,1}=\cP^+_k$ and $\cP^-_{k,1}=\cP^-_k.$

We recall the definition of viscosity sub and supersolution. First of all, we say that a function $\varphi$ touches from above (below) at $x$ if, for a small $r>0,$
\[\varphi(x)=u(x)\,\,\,\, \text{and}\,\,\,\, u(y) \leq (\geq) \varphi(y)\,\,\,\, \text{for all}\,\,\,\, y\in \sB_r(x).\]
\begin{definition}\label{viscosity defn}
    A function $u \in USC(\Rd) \cap L^{\infty}(\Rd)$ (resp. $u\in LSC(\Rd) \cap L^{\infty}(\Rd) $) is said to be a viscosity subsolution (resp. supersolution) to \eqref{Eq-1} if whenever $\varphi$ touches $u$ from above (resp. below) for some bounded test function $\varphi \in C^2(\sB_r(x)),$ then 
\begin{align*}
v=\begin{cases}
\varphi\,\,\,\, &\text{in}\,\,\, \sB_r(x)\\
u\,\,\,\, &\text{in}\,\,\, \sB^c_r(x)
\end{cases}
\end{align*}
satisfies $\cL v(x)+C_0|Dv(x)| \geq -K$ (resp. $\cL v(x)-C_0|Dv(x)| \leq K$).
\end{definition}

\subsection{Auxiliary lemmas } We collect some preliminary results here. The first result is the interior $C^{1,\gamma}$ regularity for the scaled operator $\cL^{s}.$
\begin{lemma}\label{TH2.1}
Let  $0<s\leq 1$  and $u \in L^{\infty}(\Rd)\cap C(\Rd)$ solves the operator inequalities
\begin{equation}\label{eqq2.1}
\begin{aligned}
    \cL^s[x,u]+C_0s|Du(x)| \geq -K\,\,\,\, &in\,\,\, \sB_2,\\
    \cL^s[x,u]-C_0s|Du(x)| \leq K\,\,\,\, &in\,\,\, \sB_2,
 \end{aligned}
\end{equation}
in the viscosity sense. Then there exist constants $0<\gamma <1$ and $C>0$ independent of $s$, such that 
\begin{align*}
    ||u||_{C^{1, \gamma}(\sB_1)} \leq C\Big(||u||_{L^{\infty}(\Rd)}+ K\Big),
\end{align*}
where $\gamma$ and $C$ depend only on $d, \lambda, \Lambda, C_0 $ and $\int_{\Rd} (1 \wedge |y|^2) k(y) \D{y} $. 
\end{lemma}
\begin{proof}
The proof essentially uses the approximation arguments for nonlocal equations \cite{CS11} and it is postponed to \cref{appendix}.
\end{proof}
Now we present a maximum principle type result similar to \cite[Theorem 5.2]{BM21}. We report the proof here for reader's convenience.
\begin{lemma}\label{A3.1}
Let u be a bounded function on $\Rd$ which
is in $USC(\overline{\Omega})$ and satisfies $\cP^{+} u + \cP_k^{+} u + C_0 |Du| \geq 0$ in $\Omega$. 
Then we have $\sup_{\Omega}  u \leq  \sup_{\Omega^c} u.$ 
\end{lemma}
\begin{proof}
From \cite[Lemma 5.5]{M19} we can find a non-negative function 
$\chi\in \cC^2(\overline{\Omega})\cap\cC_b(\Rd)$ satisfying
$$\cP^{+} \chi+\cP_k^{+} \chi + C_0 |D\chi| \leq -1\quad \text{in}\; \Omega.$$
Note that, since $\chi\in C^2(\overline{\Omega})$,
the above inequality holds in the classical sense. 
For $ \varepsilon> 0$,
we let $\phi_M$ to be
$$
\phi_M (x) = M + \varepsilon \chi.
$$
Then $\cP^{+}\phi_{M}(x_0)+\cP^{+}_k\phi_{M} +C_0|D\phi_{M}|\leq - \varepsilon$ in $\Omega$.

Let $M_0$ be the smallest value of $M$ for which $\phi_M \geq u$ in $\Rd$. We show that
$M_0 \leq \sup_{\Omega^c } u$.  Suppose, to the contrary, that $M_0 >  \sup_{\Omega^c } u$. Then there must be a point $x_0 \in \Omega$ for which $u(x_0) = \phi_{M_0}(x_0).$ Otherwise using the upper semicontinuity of $u,$ we get a $M_1<M_0$ such that $\phi_{M_1}\geq u$ in $\Rd,$ which contradicts the minimality of $M_0$.  Now $\phi_{M_0}$ would touch $u$ from above at
$x_0$ and thus, by the definition of the viscosity subsolution, we would have that  $\cP^{+}\phi_{M_0}(x_0)+\cP^{+}_k\phi_{M_0}(x_0) +C_0|D\phi_{M_0}(x_0)| \geq 0$. This leads to a contradiction. Therefore,  
$M_0 \leq \sup_{\Omega^c } u$ which implies that
for every $x \in \Rd$
$$
u \leq \phi_{M_0} \leq M_0 + \varepsilon\, \sup_{\Rd}\chi \leq \sup_{\Omega^c} u +\varepsilon\, \sup_{\Rd}\chi.
$$
The result follows by taking $\varepsilon \rightarrow 0$.
\end{proof}
\begin{remark}\label{rem2.1}
    Although we have the above maximum principle, one cannot simply compare two viscosity sub and supersolutions for the operator \eqref{eq1.1}. More precisely, if $u,v$ are bounded functions and $u\in USC(\Rd), v\in LSC(\Rd)$ satisfy
    $$\cL u +C |Du|\geq f \,\,\,\,\text{and}\,\,\,\,\cL v + C |Dv|\leq g \,\,\,\, \text{in}\,\,\,\, \Omega $$ in viscosity sense for two continuous functions $f$ and $g,$  and for some $C \geq 0,$ then $\cL (u-v) + C |D(u-v)|\geq f-g$ may not always hold true in $\Omega$. However, if one of them is $C^2,$ then we have 
    \[ \cP^{+}(u-v)+\cP^{+}_k(u-v) +C |D(u-v)|\geq f-g\,\,\,\, \text{in}\,\,\, \Omega.\] Indeed, without loss of generality,  let us assume $v \in C^2(\Omega)$ and $\varphi$ be a $C^2$ test function that touches $u-v$ at $x \in \Omega$ from above then clearly $\varphi+v$ touches $u$ at $x$ from above. By definition of viscosity subsolution  we have 
    $\cL (\varphi+v)(x) +C |D(\varphi+v)(x)|\geq f(x),$ which implies $$\cP^+ \varphi (x) +\cP^+_k \varphi(x)+\cL v(x)+C |D \varphi(x)|+C |D v(x)|\geq f(x)$$ and hence we obtain $$\cP^+\varphi(x) +\cP^+_k\varphi(x)+C |D \varphi(x)|\geq f(x)-g(x).$$
\end{remark}

\section{Global Lipschitz regularity}\label{Section 3}In this section we establish the Lipschitz regularity of the solution $u$ up to the boundary. We start by showing that the distance function $\delta(x)$ is a barrier to $u.$
\begin{lemma}\label{Barrier1}
Let $\Omega $ be a bounded $C^{1,1}$ domain in $\Rd$ and $u$ be a continuous function which solves \eqref{Eq-1} in the viscosity sense. Then there exists a constant  $C$ which  depends only on $d, \lambda, \Lambda, C_0, \diam (\Omega)$, radius of exterior ball and $\int_{\Rd}(1\wedge |y|^{2}) k(y) \D{y}$, such that
\begin{equation}\label{EL2.1A}
|u(x)|\leq C K\delta(x)\quad \text{for all}\; x\in \Omega.
\end{equation}
\end{lemma}
\begin{proof}
First we show that 
\begin{equation}\label{EL2.1B}
| u(x)|\leq \kappa\, K \quad x\in \Rd,
\end{equation}
for some constant $\kappa$. From \cite[Lemma 5.5]{M19}, there exists a non-negative function 
$\chi\in C^2(\overline{\Omega})\cap C_b(\Rd)$, with $\inf_{\Rd}\chi>0$, satisfying
$$
\cP^+ \chi  + \cP^+_k \chi+  C_0 |D \chi |  \leq -1\quad \text{in}\;\,\, \Omega.
$$
We define $\psi  =  K \chi $ which gives  that $\inf_{\Rd}\psi\geq 0$ and
\begin{equation*}\label{EL2.1C}
\cP^+ \psi  +\cP^+_k \psi +C_0 |D \psi |  \leq -K \quad \text{in}\;\,\, \Omega.
\end{equation*}
Then by using \cref{rem2.1}, we get
\[ \cP^+(u-\psi)+\cP^+_k(u-\psi)+C_0|D(u-\psi)|\geq 0.\] Now applying \cref{A3.1} on $u-\psi$ we obtain $$\sup_{\Omega} ( u-\psi) \leq  \sup_{\Omega^c} (u-\psi) \leq 0.$$ Note that in the second inequality above we used $u=0$ in $\Omega^c.$
This proves that $u\leq \psi$ in $\Rd$.
Similar calculation using $-u$ will also give us $-u\leq \psi$
in $\Rd$. Thus 
\begin{equation*}
| u | \leq \sup_{\Rd} |\chi|\, K \quad \text{in}\; \Rd,
\end{equation*}
which gives \eqref{EL2.1B}.

Now we shall prove \eqref{EL2.1A}. Since $\partial \Omega$ is 
 $C^{1,1}$, $\Omega$ satisfies a uniform exterior ball condition. Let $r_\circ$ be a radius satisfying uniform exterior ball condition. From \cite[Lemma~5.4]{M19} there exists a bounded, Lipschitz continuous function
$\varphi$, Lipschitz constant being $r_\circ^{-1}$, satisfying
\begin{align*}
\varphi &= 0  \quad \text{in} \quad \bar{\sB}_{r_\circ} , 
\\
\varphi &> 0 \quad \text{in} \quad \bar{\sB}_{r_\circ}^c , 
\\
\varphi &\geq \varepsilon \quad \text{in} \quad  \sB_{(1+\delta)r_\circ}^c   , 
\\
\cP^{+} \varphi + \cP^{+}_{k} \varphi + C_0 |D \varphi|  &\leq -1   \quad \text{in} \quad \sB_{(1+\delta)r_\circ}\setminus \bar{\sB}_{r_\circ},
\end{align*}
for some constants $\varepsilon, \delta$, dependent on $C_0 , d, \lambda,\Lambda, d$ and $\int_{\Rd}(1 \wedge |y|^2 ) {k}(y)\D{y}$.  Furthermore, $\varphi$
is $C^2$ in $\sB_{(1+\delta)r_\circ}\setminus \bar{\sB}_{r_\circ}$.
For any point $y\in\partial \Omega$,
we can find another point $z\in \Omega^c$ such that $\overline{\sB}_{r_\circ}(z) \subset \Omega^c$ touches
$\partial \Omega$ at $y$. 
Let 
$w(x)=\varepsilon^{-1}\kappa K  \varphi(x-z)$.  Also, $\cP^{+}(w) + \cP^{+}_{k}(w) + C_0 | D w|  \leq -K $. Then by using \cref{rem2.1} we have
 $$
 \cP^{+}(u-w)+\cP^{+}_k(u-w ) + C_0 | D( u -w)  |  \geq 0  \quad \text{in}\; \sB_{(1+\delta) r_\circ}(z)\cap \Omega.
$$  
Since, by \eqref{EL2.1B} $u-w\leq 0$ in $(\sB_{(1+\delta) r_\circ}(z)\cap \Omega)^c$, applying \cref{A3.1} on $u-w,$
 it follows that 
$u(x)\leq w(x)$ in $\Rd$. Repeating a similar calculation for $-u$, 
we can conclude that $|u(x)|\leq w(x)$ in $\Rd$.
Since this relation holds for any $y\in \partial \Omega,$
taking $x\in \Omega$ with $\dist(x, \partial \Omega)< r_\circ$, one can find
$y\in \partial \Omega$ satisfying $\dist(x, \partial \Omega)=\abs{x-y}< r_\circ.$ Then using the previous estimate
we would obtain
$$|u(x)|\leq \varepsilon^{-1} \kappa K  \varphi(x-z)
\leq \varepsilon^{-1} \kappa K (\varphi(x-z)-
\varphi(y-z))\leq \varepsilon^{-1} \kappa K \,
r_{\circ}^{-1}\dist(x, \partial \Omega),
$$
which gives us \eqref{EL2.1A}.
\end{proof}
Now we are ready to prove that $u\in C^{0,1}(\Rd)$.
\begin{proof}[Proof of \cref{T1.1}]
Let $x_0 \in \Omega$ and $s\in (0,1]$ be such that
$2s= \dist(x_0, \partial \Omega)\wedge 1$.  Without
loss of any generality, we assume $x_0=0$.
Define $v(x) = u(sx)$ in $\Rd$.
Using \cref{Barrier1} we already have $|u(x)| \leq C_1K\delta(x)$, from that one can deduce
   \begin{equation}\label{ET2.1B}
   | v(x) | \leq C_1\,K 
    s (1+\abs{x})\quad \quad \text{for all} \,\,\; x \in \Rd,
   \end{equation}
for some constant $C_1$ independent of $s$. We recall the scaled operator
$$\cI_{\theta \nu }^s [x,f]:=\int_{\Rd} (f(x+y) -f(x) - \Ind_{\sB_{\frac{1}{s}}}(y) \grad f(x) \cdot y   )s^{d+2}N_{\theta \nu}(sx, sy) \D{y}.$$

To compute $\cL^s [x, v] + C_0 s |D v(x)|$ in $\sB_2$, first we observe that $D^2 v(x) = s^2 D^2 u(sx)$ and $Dv(x) = s Du(sx)$. Also
\begin{align*}
   \cI_{\theta \nu }^s[x,v] &= s^2 \int_{\Rd} (v(x+y) -v(x)-\Ind_{\sB_{\frac{1}{s}}}(y)\grad v(x) \cdot y  ) N_{\theta \nu }(sx,sy) s^d \D{y} 
   \\
   &=s^2 \int_{\Rd} (u(sx+sy) -u(sx)-\Ind_{\sB_{1}}(sy)\grad u(sx) \cdot sy  ) N_{\theta \nu }(sx,sy) s^d \D{y} = s^2 \cI_{\theta \nu } [sx,u].
   \end{align*}
Thus, it follows from \eqref{Eq-1} that 
\begin{equation}\label{ET2.1C}
\begin{split}
\cL^s [x,v]  + C_0 s |D v(x)| &\geq -Ks^2 \quad \text{in} \quad \sB_2,
\\
\cL^s [x,v] - C_0 s|D v(x)| &\leq Ks^2 \quad \text{in} \quad \sB_2.
\end{split}
\end{equation}
Now consider a smooth cut-off function $\varphi, 0\leq\varphi\leq 1$, satisfying
\begin{equation*}
    \varphi=
    \begin{cases}
    1\quad \text{in}\; \sB_{3/2},\\
    0\quad \text{in}\; \sB^c_2.
    \end{cases}
\end{equation*}
Let $w = \varphi v$. Clearly,  $((\varphi-1)v)(y) = 0$ for all $y \in \sB_{3/2}$, which gives $D((\varphi-1)v) = 0 $ and $D^2 ((\varphi-1)v)=0$ in $x \in \sB_{3/2}$. Since $w = v + (\varphi-1)v$, from \eqref{ET2.1C} we obtain
\begin{equation}\label{ET2.1D}
\begin{split}
\cL^s [x,w]  + C_0 s |D w(x)| &\geq -Ks^2 - |\sup_{\theta \in \Theta} \inf_{\nu \in \Gamma} \cI_{\theta\nu }^s [x,(\varphi-1)v) ]| \quad \text{in} \quad \sB_1,
\\
\cL^s [x,w] - C_0 s|D w(x)| &\leq Ks^2  + |\sup_{\theta \in \Theta} \inf_{\nu \in \Gamma} \cI_{\theta\nu }^s [x,(\varphi-1)v) ]|  \quad \text{in} \quad \sB_1.
\end{split}
\end{equation} Again, since $(\varphi-1)v=0$ in $\sB_{3/2}$,
for  $x \in \sB_1, $ we have
\begin{align*}
    | \cI_{\theta \nu}^s [x,(\varphi-1)v]|&=\Big|\int_{|y|\geq 1/2} ((\varphi-1)v)(x+y)-((\varphi-1)v)(x))s^{d+2}N_{\theta\nu }(sx,sy)\D{y}\Big|
    \\
 &\leq \int_{|y|\geq 1/2} |v(x+y)|s^{d+2} N_{\theta \nu}(sx,sy)\D{y} := I. 
\end{align*}
Now we write 
\begin{align*}
    I&=\int_{1/2 \leq |y|\leq 1/s} |v(x+y)|s^{d+2} N_{\theta \nu}(sx,sy)\D{y}+\int_{|y|\geq 1/s} |v(x+y)|s^{d+2} N_{\theta \nu}(sx,sy)\D{y}\\
    &=I_{s,1}+I_{s, 2} \, .
\end{align*}
Let us first estimate $I_{s,1}.$ Since $x \in \sB_1$ and $|y| \geq \frac{1}{2} $ we have $1+|x+y| \leq 5 |y|.$ By using this estimate and \eqref{ET2.1B} we obtain 
\begin{align*}
I_{s,1} &= s^{d+2} \int_{\frac{1}{2}\leq\abs{y}\leq\frac{1}{s} } |v(x+y)| N_{\theta \nu}(sx, sy)\D{y} \\
&\leq 5C_1K \int_{\frac{1}{2}\leq\abs{y}\leq\frac{1}{s}} 
|sy|s^{d+2} k(sy)\D{y} 
\leq 5 C_1 K s \int_{\frac{s}{2}\leq\abs{z}\leq 1} |sz|k(z)\D{z} 
\\
& \leq    C_2 s \int_{\frac{s}{2}\leq\abs{z}\leq 1} |z|^2 k(z) \D{z} 
\leq    C_2 s \int_{\Rd} ( 1\wedge |y|^2) k(z) \D{z} 
 \leq C_3 s,  
\end{align*}
for some constants $C_3.$ For $I_{s,2}$, a change of variable and \eqref{EL2.1B} gives
\begin{align*}
I_{s,2}
\leq \kappa s^{2} K  \int_{ s \abs y >1 } s^d k(sy) \D{y}
&= \kappa s^{2}K  \int_{\abs y >1 }  k(y)\D{y}
\\
&\leq \kappa s^{2}K   \int_{\Rd}(1 \wedge |y|^2) k(y) \D{y} 
\leq C_4 s^{2}K 
\end{align*}
for some constant $C_4$. Therefore, putting the estimates of 
$I_{1}$ and $I_{2}$ in \eqref{ET2.1D} we obtain
\begin{equation}\label{ET2.1E}
\begin{split}
\cL^s [x,w]  + C_0 s |D w(x)| &\geq - C_5 K s \quad \text{in} \quad \sB_1,
\\
\cL^s [x,w] - C_0 s|D w(x)| &\geq  C_5 K s \quad \text{in} \quad \sB_1,
\end{split}
\end{equation}
for some constant $C_5$. Now applying \cref{TH2.1}, from \eqref{ET2.1E}
we have
\begin{equation}\label{ET2.1F}
 \norm{v}_{C^{1} (\sB_{\frac{1}{2}}) } \leq C_6 \Big( \norm{v}_{L^{\infty}(\sB_2)}+ s  K \Big)   
\end{equation}
for some constant $C_6$. From
\eqref{ET2.1B} and \eqref{ET2.1F} we then obtain
\begin{equation}\label{ET2.1G}
\sup_{y\in \sB_{s/2}(x), y\neq x}\frac{|u(x)-u(y)|}{|x-y|}
\leq C_7 K,
\end{equation}
for some constant $C_7$.

Now we can complete the proof. Note that if $|x-y|\geq \frac{1}{8}$, then 
$$\frac{|u(x)-u(y)|}{|x-y|}\leq 2\kappa K,$$
by \eqref{EL2.1B}. So we consider $|x-y|< \frac{1}{8}$. If 
$|x - y| \geq 8^{-1} ( \delta(x) \vee \delta(y))$, then using
\cref{Barrier1} we get
 $$
\frac{| u(x) - u(y)|  }{|x - y|}  
\leq 4CK (\delta(x)+\delta(y)) ( \delta(x) \vee \delta(y))^{-1}  \leq 8CK.
 $$
Now let $|x - y| < 8^{-1} \min\{ \delta(x) \vee \delta(y), 1\}$.
Then either $y\in \sB_{\frac{\delta(x)\wedge 1}{8}}(x)$ or $x\in \sB_{\frac{\delta(y)\wedge 1}{8}}(y)$. Without loss of generality, we suppose
$y\in \sB_{\frac{\delta(x)\wedge 1}{8}}(x)$. 
From \eqref{ET2.1G} we get
$$\frac{| u(x) - u(y)|  }{|x - y|}\leq C_7 K.$$
This completes the proof.
\end{proof}

\section{Fine boundary regularity}\label{Section 4}
The aim of this section is to prove \cref{T1.2}. Since $u$ is Lipschitz, \eqref{Eq-1} can be written as
\[|\cL u| \leq CK\,\,\,\, \text{in}\,\,\, \Omega, \,\,\,\,\,\, \text{and}\,\,\,\, u=0\,\,\,\, \text{in}\,\,\, \Omega^c.\]We start by constructing subsolutions which will be useful later on to prove oscillation lemma.
\begin{lemma}\label{L2.2} 
There exists a constant $\tilde\kappa$, which depends only on $d, \lambda, \Lambda ,\int_{\Rd} (1\wedge|y|^2) k(y)\D{y}$,  such that for any $r\in (0, 1]$,  we have a bounded radial function $\phi_r \in C^2(\sB_{4r}\setminus \bar{\sB}_r)$ satisfying
\begin{equation*}
\begin{cases}
\cP^{-} \phi_r + \cP^{-}_{k} \phi_r \geq 0   &\text{in} \;\; \sB_{4r} \setminus \bar{\sB}_r, 
\\
0 \leq \phi_r \leq \tilde\kappa r  & \text{in} \;\; \sB_r , 
\\
\phi_r \geq \frac{1}{\tilde\kappa} (4r - \abs x) & \text{in} \; \; \sB_{4r} \setminus \sB_r , 
\\
\phi_r \leq 0 & \text{in} \;\;  \Rd \setminus \sB_{4r}.
\end{cases}
\end{equation*}
\end{lemma}

\begin{proof}
We use the same subsolution constructed in \cite{BMS22} and show that it is indeed a subsolution with respect to minimal Pucci operators.
Fix $r\in(0, 1]$  and define $v_r(x) = \E^{- \eta q(x)} - \E^{-  \eta (4 r)^2}$, where $q(x) = |x|^2 \wedge 2(4 r)^2$ and
$\eta >\frac{1}{r^2}$.  Clearly, $1\geq v_r(0) \geq v_r(x)$ for all $x \in \Rd$. Thus,  using the fact that $1-e^{-\xi}\leq \xi$ for all $\xi\geq 0$ we have

\begin{equation}\label{EL2.2A}
v_r(x) \leq 1- \E^{-\eta(4r)^2}\leq \eta(4r)^2,
\end{equation}
Again for $x \in \sB_{4r} \setminus \sB_r$, we have that
\begin{align}\label{EL2.2B}
v_r(x) = \E^{-\eta (4r)^2 } ( \E^{\eta ((4r)^2 - q(x))} -1  )
&\geq \eta \E^{-\eta (4r)^2 }  ((4r)^2 - |x|^2) \nonumber
\\
& = \eta \E^{-\eta (4r)^2 }  (4r + |x|)  (4r - |x|) 
 \geq 5  \eta r  \E^{-\eta (4r)^2 }   (4r - |x|).
\end{align}
Fix $x\in \sB_{4r} \setminus \bar{\sB}_r$. We start by estimating the local minimal Pucci operator $\cP^{-}$ of $v$.  Using rotational symmetry we may always assume $x = (l , 0, \cdots,0 )$ Then
$$\partial_{i} v_r(x) = -2 \eta \E^{- \eta \vert x  \vert^2} x_i = \begin{cases}  -2 \eta \E^{- \eta \vert x  \vert^2} l & i=1,\\
0 & i \ne 1 \end{cases}  $$
and
\begin{align*}
\partial_{ij} v_r(x) &= \begin{cases} 4\eta^2 x_i^2  \E^{- \eta \vert x  \vert^2} - 2\eta \E^{- \eta \vert x  \vert^2} & i=j, \\
4\eta^2 x_ix_j \E^{- \eta \vert x  \vert^2} & i \ne j .
\end{cases} \\
&= \begin{cases} 4\eta^2 l^2  \E^{- \eta \vert x  \vert^2} - 2\eta \E^{- \eta \vert x  \vert^2} & i=j =1,\\
 - 2\eta \E^{- \eta \vert x  \vert^2} & i=j \ne 1, \\
0 & i\ne j.
\end{cases}
\end{align*}
By the above computation,  for any $x\in \sB_{4r} \setminus \bar{\sB}_r,$ we have
\begin{align*}
\cP^{-}v_r(x) &=  \lambda 4\eta^2 l^2  \E^{- \eta \vert x  \vert^2} - \lambda 2\eta \E^{- \eta \vert x  \vert^2} - \Lambda(d-1) 2\eta \E^{- \eta \vert x  \vert^2} \\
&= \lambda 4\eta^2 l^2  \E^{- \eta \vert x  \vert^2} +(\Lambda-\lambda)2\eta e^{-\eta|x|^2} - d \Lambda  2\eta \E^{- \eta \vert x  \vert^2}\\
&\geq \lambda 4\eta^2 l^2  \E^{- \eta \vert x  \vert^2}- d \Lambda  2\eta \E^{- \eta \vert x  \vert^2}.
\end{align*}
Now to determine nonlocal minimal Pucci operator, 
using the convexity 
of exponential map we get,  
\begin{align*}
&e^{-\eta |x+y|^2}-e^{-\eta |x|^2}+2 \eta \Ind_{\{|y|\leq 1\}} y\cdot x e^{-\eta |x|^2}
\\
&\qquad\geq -\eta e^{-\eta | x|^2}\left(|x+y|^2-|x|^2-2\Ind_{\{|y|\leq 1\}}
y\cdot x\right).
\end{align*}
Since $\cP^{-}_k v_r=\cP^{-}_k(v_r + e^{-\eta (4r)^2})$ and using above inequality we obtain
\begin{align*}
\cP^{-}_k(e^{-\eta q(\cdot)})(x)&=-\int_{\Rd}\left(e^{-\eta q(x+y)}-e^{-\eta q(x)}-\Ind_{\sB_1}(y)\grad e^{-\eta q(x)}\cdot y\right)^{-}k(y)dy\\
&\geq 
-\eta e^{-\eta|x|^2}\int_{|y| \leq r} \Big( |x+y|^2 - |x|^2  - 2
y\cdot x  \Big) k(y) \D{y} \\
&- \int_{r < |y| \leq 1} \left| e^{-\eta (|x|^2 + 2(4r)^2)} - e^{-\eta |x|^2} + 2\eta y\cdot x e^{-\eta |x|^2} \right| k(y) \D{y} \\
 &- \int_{|y| >1} \left| e^{-\eta (|x|^2 +2(4r)^2)} - e^{-\eta |x|^2} \right|  k(y) \D{y}
   \\
 &\geq -\eta e^{-\eta|x|^2} \left[ \int_{\abs{y} <r}   \abs{y}^2 k(y) \D{y} + \int_{r < |y| \leq 1} \left( \left| \frac{ 1-  e^{-2\eta(4r)^2 } }{\eta}\right| + 2|x \cdot y| \right) k(y)\D{y} \right] \\
 &-\eta e^{-\eta|x|^2} \int_{|y| >1} \left| \frac{ 1-  e^{-2\eta(4r)^2 } }{\eta} \right| k(y) \D{y} \\
   &\geq -\eta e^{-\eta|x|^2} \left[ \int_{\abs{y} <r}   \abs{y}^2 k(y) \D{y} + \int_{ r < \abs{y} \leq 1} 4^3  \abs{y}^2 k(y) \D{y} + \int_{\abs{y} >1}   2(4r)^2 k(y) \D{y} \right] \\
&\geq -\eta e^{-\eta|x|^2} 4^3 \int_{\Rd} (1 \wedge |y|^2) k(y) \D{y},
\end{align*}
where in the second line we used $|x+y|^2\wedge 2(4r)^2\leq |x|^2+ 2(4r)^2$.  Combining the above estimates we obtain that, for $x\in \sB_{4r}\setminus\bar{\sB}_r$,
\begin{align*}
\mathcal{P}^{-} v_r(x) + \mathcal{P}_k^{-} v_r(x) &\geq \eta \E^{- \eta \abs{x}^2} \Bigl[4\eta  \lambda |x|^2  - 2 d \Lambda - 4^3 \int_{\Rd} (1 \wedge |y|^2) k(y) \D{y} \Bigr] \\
& \geq \eta \E^{- \eta \abs{x}^2} \Bigl[4\eta  \lambda r^2  - 2 d \Lambda - 4^3 \int_{\Rd} (1 \wedge |y|^2) k(y) \D{y} \Bigr].
\end{align*}
Thus, finally letting $\eta= \frac{1}{ \lambda r^2}(2d\Lambda + 4^3 \int_{\Rd} (1 \wedge |y|^2) k(y) \D{y} )$, we obtain 
$$\cP^{-}v_r+\cP^{-}v_r>0\quad \text{in}\; \sB_{4r}\setminus\bar{\sB}_r.$$ Note that the final choice of $\eta$ is admissible since $\frac{1}{ \lambda r^2}(2d\Lambda + 4^3 \int_{\Rd} (1 \wedge |y|^2) k(y) \D{y} )>\frac{1}{r^2}.$
Now set $\phi_r= r v_r$ and the result follows from \eqref{EL2.2A}-\eqref{EL2.2B}.
\end{proof}
Next we prove a \textit{Harnack type} inequality which will play an important role in subsequent analysis.
\begin{theorem}[Harnack type inequality]\label{T5.1}
Let $s\in (0,1],$ $\alpha^{\prime}=1\wedge(2-\alpha)$ and $u$ be a continuous  non-negative function satisfying
$$  \mathcal{P}^{-} u + \mathcal{P}^{-}_{k,s} u \leq C_0 s^{1+\alpha^{\prime}}, \quad  \mathcal{P}^{+} u + \mathcal{P}^{+}_{k,s} u \geq -C_0 s^{1+\alpha^{\prime}}  \quad \text{in}\; \sB_2.$$ Furthermore if $\sup_{\Rd} u \leq M_0$ and $u(x) \leq M_0 s(1+|x|)$ for all $x\in \Rd,$
then 
$$u(x) \leq  C ( u(0) + (M_0\vee C_0)s^{1+\alpha^{\prime}})$$
 for every $x \in \sB_{\frac{1}{2}}$ and for some constant
$C $ which only depends on  $\lambda, \Lambda, d,  \int_{\mathbb{R}^d} (1 \wedge |y|^{\alpha}) k(y)\D{y} .$
\end{theorem}
\begin{proof}
    Dividing by $u(0)+(M_0 \vee C_0)s^{1+\alpha^{\prime}},$ it 
    can be easily seen that $\sup_{\Rd}|u| \leq s^{-(1+\alpha^{\prime})}$ and $|u(x)| \leq s^{-\alpha^{\prime}}(1+|x|)$ for all $x\in \Rd$  and $u$ satisfies
    \begin{equation*}
    \begin{aligned}
       &\mathcal{P}^{-} u + \mathcal{P}^{-}_{k,s} u \leq 1, \\
       &\mathcal{P}^{+} u + \mathcal{P}^{+}_{k,s} u \geq -1.
    \end{aligned}
    \end{equation*}
Fix
$\varepsilon > 0$ (see $\epsilon_3$ in \cite[Corollary~3.13]{M19}) and
let $\gamma = \frac{d}{\varepsilon}$. 
Let
$$
t_0\df\min\left\{t\; :\; u(x) \leq h_{t}(x) \df t(1 - \abs x)^{- \gamma} \;\; \text{for all} \; x \in \sB_1\right\}.
$$
Clearly this set is nonempty since $u(0) \leq 1$, thus $t_0$ exists. Let $x_0 \in \sB_1$ be such that $u(x_0) = h_{t_0}(x_0)$. Let  
$\eta = 1 - \vert x_0 \vert$ be the distance of $x_0$ from $\partial \sB_1$. For $r = \frac{\eta}{2}$  and $x\in \sB_r(x_0),$ we can write
\[\sB_r(x_0)=\left\{u(x) \leq \frac{u(x_0)}{2}\right\}\cup \left\{u(x) >\frac{u(x_0)}{2}\right\}:=A\cup\tilde{A}.\] Our goal is to estimate $|\sB_r(x_0)|$ in terms of $|A|$  and $|\tilde{A}|.$ Proceeding this way, we show that $t_0<C$ for some universal $C$ which, in turn, implies that $u(x) < C (1 -  \abs{x})^{-\gamma}$.
This would prove our result.
Next, Using \cite[Corollary~3.13]{M19} we obtain
$$
\vert \tilde{A} \cap \sB_1 \vert \leq C \bigg\lvert \frac{2}{u(x_0)} \bigg\rvert^{\varepsilon} \leq C t_0^{-\varepsilon}
\eta^d\,,
$$
whereas $\vert \sB_r \vert = \omega_d (\eta/2)^d$. In particular,
\begin{equation}\label{ET5.1A0}
\big\lvert \tilde{A}  \cap \sB_r(x_0)  \big\rvert \leq 
C t_0^{- \varepsilon} \vert \sB_r \vert .
\end{equation}
So if $t_0$ is large,
$\tilde{A}$ can cover only a small portion of $\sB_r(x_0)$. We shall show that for some $\delta>0$,
independent of $t_0$ we have 
$$ \vert A \cap \sB_r(x_0) \vert \leq  (1- \delta) \vert \sB_r \vert ,$$
which will provide an upper bound on $t_0$ completing the proof. We start by estimating $ \vert A \cap \sB_{\uptheta r} (x_0) \vert $ for $ \uptheta > 0$ small. For every $ x \in \sB_{\uptheta r}(x_0)$ we have
$$
u(x) \leq\, h_{t_0} (x) \,\leq\, t_0 \left( \frac{2\eta - \uptheta \eta}{2} \right)^{- \gamma} 
\,\leq\, u(x_0) \left( 1 - \frac{\uptheta }{2} \right)^{- \gamma},
$$
with $\left( 1 - \frac{\uptheta }{2} \right)$ close to 1.
Define
$$
v(x) \df \left( 1 - \frac{\uptheta }{2} \right)^{- \gamma} u(x_0) - u(x).
$$
Then we get $v \geq 0$ in $\sB_{ \uptheta r}(x_0)$ and also $\mathcal{P}^{-}v+\mathcal{P}^{-}_{k,s} v \leq 1$ as 
$\mathcal{P}^{+}u+\mathcal{P}^{+}_{k,s} u \geq -1.$

We would like to apply \cite[Corollary~3.13]{M19} to $v,$ but $v$ need not be non-negative in the whole of $\Rd$. Thus, we consider the positive part of $v,$ i.e., $w=v^+$ and find an upper bound of $\mathcal{P}^{-}w+\mathcal{P}^{-}_{k,s}w$. Since $v^-$ is $C^2$ in $\sB_{\frac{\uptheta r}{4}}(x_0),$ we have
\begin{align}\label{eqq4.4}
\mathcal{P}^{-}w(x)+\mathcal{P}^{-}_{k,s}w(x) \leq [\mathcal{P}^{-}v(x)+\mathcal{P}^{-}_{k,s}v(x)]+[\mathcal{P}^{+}v^{-}(x)+\mathcal{P}^{+}_{k,s}v^{-}(x)] \leq 1+ \mathcal{P}^{+}v^{-}(x)+\mathcal{P}^{+}_{k,s}v^{-}(x).
\end{align}
Also, using $v^{-}(x)=Dv^{-}(x)=D^2v^{-}(x)=0$ for all $x \in \sB_{\frac{\uptheta r}{4}}(x_0),$ we
get
\begin{align}\label{eqq4.5}
\mathcal{P}^{+}v^{-}(x)+\mathcal{P}^{+}_{k,s}v^{-}(x)=\int_{\Rd\cap \{v(x+y) \leq 0\}}v^{-}(x+y)s^{d+2}k(sy)\D{y}.
\end{align}
Now plugging \eqref{eqq4.5} into \eqref{eqq4.4}, for all $x \in \sB_{\frac{\uptheta r}{4}}(x_0)$ we obtain
\begin{align*}
&\mathcal{P}^{-}w(x)+\mathcal{P}^{-}_{k,s}w(x)
\leq 1+\int_{\Rd \setminus \sB_{\frac{\uptheta r}{2}}(x-x_0)} \left(u(x+y)-\left( 1 - \frac{\uptheta }{2} \right)^{- \gamma} u(x_0) \right)^{+} s^{d+2}k(sy) \D{y}\\
&\leq 1+\int_{\Rd \setminus \sB_{\frac{\uptheta r}{2}}(x-x_0)} \left|u(x+y)\right|s^{d+2}k(sy)\D{y}+\int_{\Rd \setminus \sB_{\frac{\uptheta r}{2}}(x-x_0)} \left|\left( 1 - \frac{\uptheta }{2} \right)^{- \gamma} u(x_0)\right| s^{d+2}k(sy)\D{y}\\
&\leq 1+\int_{\Rd \setminus \sB_{\frac{\uptheta r}{4}}} \left|u(x+y)\right|s^{d+2}k(sy)\D{y}+\int_{\Rd \setminus \sB_{\frac{\uptheta r}{4}}} \left|\left( 1 - \frac{\uptheta }{2} \right)^{- \gamma} u(x_0)\right| s^{d+2}k(sy)\D{y}
:=1+I_1+I_2.
\end{align*}
\textbf{Estimate of $I_1$:} Let us write
\begin{equation*}
\begin{aligned}
I_1&=
\int_{\frac{\uptheta r}{4} \leq |y| \leq \frac{1}{s}}\left|u(x+y)\right|s^{d+2}k(sy) \D{y} + \int_{|y| \geq \frac{1}{s}} \left|u(x+y)\right|s^{d+2}k(sy)\D{y}:=I_{11}+I_{12}.
\end{aligned}
\end{equation*}
Simply using change of variable and  $\sup_{\Rd}|u| \leq s^{-(1+\alpha^{\prime})}, $ we obtain
\[I_{12} \leq \int_{|z|\geq 1}k(z) \D{z}.\]
Now we estimate $I_{11}$ using  $|u(x)| \leq s^{-\alpha^{\prime}}(1+|x|)$ for all $x\in \Rd$.
\begin{align*}
I_{11} &\leq \int_{\frac{\uptheta r}{4} \leq |y|\leq \frac{1}{s}} \left(1+|x+y|\right)s^{d+2-\alpha^\prime}k(sy)\D{y}\\
&\leq \frac{5}{4} \int_{\frac{\uptheta r}{4} \leq |y| \leq \frac{1}{s}}s^{d+2-\alpha^\prime}k(sy)\D{y}+ \int_{\frac{\uptheta r}{4} \leq |y| \leq \frac{1}{s}}s^{d+2-\alpha^\prime}|y|k(sy)\D{y}\, .
\end{align*}
 We consider two cases. First consider the case $\alpha^{\prime}=1$ so $\alpha \leq 1$. This implies 
$$
I_{11} \leq\frac{5}{4}\int_{\frac{\uptheta r s}{4} \leq |z|\leq 1}sk(z)\D{z}+\int_{\frac{\uptheta r s}{4} \leq |z|\leq 1} |z|k(z)\D{z} \leq 6 (\uptheta r)^{-1}\int_{\Rd}(1\wedge |z|^\alpha)k(z)\D{z}.
$$
Now consider the case $\alpha^{\prime}=2-\alpha,$ and hence $\alpha>1.$ In this case
\begin{align*}
I_{11} &\leq \frac{5}{4} \int_{\frac{\uptheta r }{4} \leq |y|\leq \frac{1}{s}} s^{\alpha} s^d k(sy)\D{y}+\int_{\frac{\uptheta r }{4} \leq |y|\leq \frac{1}{s}} s^{\alpha-1}|sy| s^d k(sy)\D{y}\\
&= 5 \cdot 4^{\alpha-1} (\uptheta r)^{-\alpha} \int_{\frac{\uptheta r s}{4} \leq |z|\leq 1} \left( \frac{\uptheta r}{4} s\right)^{\alpha}k(z)\D{z}+\left( \frac{\uptheta r}{4} \right)^{1-\alpha} \int_{\frac{\uptheta r s}{4} \leq |z|\leq 1} \left( \frac{\uptheta r}{4} s\right)^{\alpha-1}|z|k(z)\D{z}\\
&\leq C(\uptheta r)^{-2} \int_{\Rd}(1\wedge |z|^{\alpha})k(z)\D{z}\, .
\end{align*}
Combining the estimates of $I_{11}$ and $I_{12},$ we get
\begin{align*}\label{equation3.2}
   I_1 \leq C (\uptheta r)^{-2}\int_{\Rd}(1\wedge |z|^\alpha)k(z)\D{z}. 
\end{align*} 
\textbf{Estimate of $I_2$:} If $\alpha^{\prime}=1,$ then $\alpha \leq 1$ and using  $|u(x_0)| \leq s^{-\alpha^{\prime}}(1+|x_0|)$ we have
\begin{align*}
&I_2:=\int_{\Rd\setminus \sB_{\frac{\uptheta r}{4}} }\left|\left( 1 - \frac{\uptheta }{2} \right)^{- \gamma} u(x_0)\right| s^{d+2}k(sy)\D{y} \leq C\int_{\Rd\setminus \sB_{\frac{\uptheta r}{4}} } s^{d+2-\alpha^{\prime}}k(sy)\D{y}\\
&=C\int_{\Rd\setminus \sB_{\frac{\uptheta r s}{4}} }sk(z)\D{z}
 \leq C\left[\int_{\frac{\uptheta r s}{4}\leq |z|\leq 1}sk(z)\D{z} +\int_{|z|\geq 1} sk(z)\D{z}\right] \\
 &\leq C\left[\frac{4}{\uptheta r}\int_{\frac{\uptheta r s}{4}\leq |z|\leq 1}|z|^{\alpha}k(z)\D{z} + \int_{|z|>1} k(z)\D{z}\right]
\leq  C (\uptheta r)^{-1} \int_{\Rd} (1 \wedge |y|^{\alpha}) k(z)\D{z}.
\end{align*}
 If $\alpha^{\prime}=2-\alpha$ then $\alpha>1.$  In that case, using similar calculation as above we have
\begin{equation*}
\begin{aligned}
&I_2:=\int_{\Rd\setminus \sB_{\frac{\uptheta r}{4}}}\left|\left( 1 - \frac{\uptheta }{2} \right)^{- \gamma} u(x_0)\right| s^{d+2}k(sy)\D{y} \leq C\int_{\Rd\setminus \sB_{\frac{\uptheta r}{4}} } s^{d+\alpha}k(sy)\D{y}\\
&=C\int_{\Rd\setminus \sB_{\frac{\uptheta r s}{4}} }s^{\alpha}k(z)\D{z}
\leq C (\uptheta r)^{-\alpha} \int_{\Rd} (1 \wedge |y|^{\alpha}) k(z)\D{z} .
\end{aligned} 
\end{equation*}
Since $\alpha\in (0,2),$ combining the above estimates we obtain
\[\mathcal{P}^{-}w + \mathcal{P}^{-}_{k,s}w \leq \frac{C}{(\uptheta r)^2}\,\,\,\, \textnormal{in}\,\,\,\, \sB_{\frac{\uptheta r}{4}}(x_0) \, . \]
Now using \cite[Corollary~3.13]{M19} for $w$ we get
\begin{align*}
  &\vert A \cap \sB_{\frac{\uptheta r}{8}} (x_0) \vert 		
 =   \bigg\lvert \bigg\lbrace   w \geq  u(x_0) ((1- \nicefrac{\uptheta}{2})^{ - \gamma}-\nicefrac{1}{2})   \bigg\rbrace  \cap \sB_{\frac{\uptheta r}{8}}(x_0)  \bigg\rvert
 \\
& \qquad \leq C (\uptheta r)^d  \left(\inf_{\sB_{\frac{\uptheta r}{8}}(x_0)}w + \frac{\uptheta r}{8} \cdot \frac{C}{(\uptheta r)^2} \right)^{\varepsilon}
   \cdot \left[ u(x_0) ((1- \nicefrac{\uptheta}{2})^{ - \gamma}-\nicefrac{1}{2})  \right]^{- \varepsilon}
  \\
 & \qquad \leq C (\uptheta r)^d  \left( \left((1-\frac{\uptheta}{2})^{-\gamma}- 1 \right) u(x_0) + \frac{\uptheta r}{8} \cdot \frac{C}{(\uptheta r)^2} \right)^{\varepsilon}
   \cdot \left[ u(x_0) ((1- \nicefrac{\uptheta}{2})^{ - \gamma}-\nicefrac{1}{2})  \right]^{- \varepsilon}
  \\
  &\quad \quad \leq C(\uptheta r)^d \Big[\left((1-\frac{\uptheta}{2})^{-\gamma}- 1 \right)+\frac{C}{8} (\uptheta r)^{-1}t_0^{-1} \Big]^{\varepsilon}\\
  & \quad \quad \leq C (\uptheta r)^d \left( \left((1- \nicefrac{\uptheta}{2})^{ - \gamma}-1 \right)^{\varepsilon} + C_0 (\uptheta r)^{-\varepsilon} t_0^{-\varepsilon}  \right)\,.
\end{align*}
Now let us choose $\uptheta > 0$ small enough (independent of $t_0$) to satisfy
$$
C (\uptheta r)^d  \left((1- \nicefrac{\uptheta}{2})^{ - \gamma}-1 \right)^{\varepsilon} \leq\;
 \frac{1}{4} \vert \sB_{\frac{\uptheta r}{8}}(x_0) \vert\,.
$$
With this choice of $\uptheta$ if $t_0$ becomes large,  then we  also have
$$
C (\uptheta r)^{d- \varepsilon}  t_0^{- \varepsilon} \leq \frac{1}{4} \vert \sB_{\frac{\uptheta r}{8}}(x_0) \vert\,,
$$
and hence
$$
\vert A \cap \sB_{\frac{\uptheta r}{8}} (x_0) \vert\leq \frac{1}{2} \vert \sB_{\frac{\uptheta r}{8}}(x_0) \vert\,. 
$$
This estimate of course implies that 
$$
\vert \tilde{A} \cap \sB_{\frac{\uptheta r}{8}} (x_0) \vert \geq C_2 \vert \sB_r \vert,
$$
but this is contradicting \eqref{ET5.1A0}. Therefore $t_0$ cannot be large and this completes the proof.
\end{proof}
\begin{corollary}\label{Col3.1}
Let $u$ satisfies the conditions of \cref{T5.1}, then the following holds.
\begin{equation*}
\sup_{\sB_{\frac{1}{4}}} u \leq C \left(\inf_{\sB_{\frac{1}{4}}} u + (M_0\vee C_0)s^{1+\alpha^{\prime}}\right)  .
\end{equation*}
\end{corollary}
\begin{proof}
Take any point $x_0 \in \overline{\sB_{\frac{1}{4}}}$ such that $u(x_0)=\inf_{\sB_{\frac{1}{4}}}u(x).$  Clearly $\sB_{\frac{1}{4}} \subset \sB_{\frac{1}{2}}(x_0).$ Defining  $\tilde{u}(x):=u(x+x_0)$ and applying \cref{T5.1} on $\tilde{u}$ we find
\[\tilde{u}(x) \leq C\left(\tilde{u}(0)+(M_0\vee C_0)s^{1+\alpha^{\prime}}\right)\,\,\, \text{in}\,\,\,\, \sB_{\frac{1}{2}}. \]
This implies
\[\sup_{\sB_{\frac{1}{4}}}u(x) \leq \sup_{ \sB_{\frac{1}{2}}(x_0)} u(x) \leq C\left(\inf_{\sB_{\frac{1}{4}}}u(x)+(M_0 \vee C_0)s^{1+\alpha^{\prime}}\right).\] This proves the claim.
\end{proof}
Now we shall prove some auxiliary lemmas which will be used to construct appropriate supersolutions that are crucial to prove the oscillation estimate. Let us denote $$\Omega_{\rho}:=\left\{x \in \Omega | \dist (x, \Omega^c) <\rho\right\}.$$ It is well known that for any $C^2$ domain, $\delta \in C^2(\Omega_{\rho}).$ Moreover, up to a suitable modification inside $\Omega$, we can extend $\delta$ to belong to  $C^2({\overline{\Omega}}).$
\begin{lemma}\label{Lem3.2}
Let $\Omega$ be a bounded $C^2$ domain in $\Rd$, 
then for any $0<\upepsilon<1,$ we have the following estimate
\begin{equation}\label{delta estimate}
\left|\cI_{\theta \nu}(\delta^{1+\upepsilon})\right| \leq C\left(1+ \Ind_{(1,2)}(\alpha)\delta^{1-\alpha}\right)\,\,\,\, \text{in}\,\,\,\, \Omega,
\end{equation}
where $C>0$ depends only on $ d, \Omega $ and $ \int_{\Rd} (1 \wedge |y|^{\alpha}) k(y) \D{y} .$
\end{lemma}
\begin{proof} 
By definition, we have
\begin{align*}
    \left| \mathcal{I}_{\theta \nu}(\delta^{1+\upepsilon})(x) \right|&\leq \int_{\Rd} \left| \delta^{1+\upepsilon}(x+y)-\delta^{1+\upepsilon}(x) -\Ind_{\sB_1}(y) y \cdot \grad \delta^{1+\upepsilon}(x)\right|k(y)\,dy.
\end{align*}
Since $\delta\in C^{0,1}(\Rd)\cap C^2(\overline{\Omega})$, using the Lipschtiz continuity of $\delta^{1+\upepsilon}$ for $|y|\leq 1$ and boundedness of $\delta^{1+\upepsilon}$ for $|y|>1$ in the right hand side of above inequality, we obtain the estimate \eqref{delta estimate} for $\alpha \in (0,1].$ 
Next consider the case $\alpha \in (1,2).$ For any $x \in \Omega$ as above we have
\begin{equation*}
\begin{aligned}
\left|\cI_{\theta \nu}(\delta^{1+\upepsilon})(x)\right| \leq &\int_{\Rd} \left|\delta^{1+\upepsilon}(x+y)-\delta^{1+\upepsilon}(x) -\Ind_{\sB_1}(y) y \cdot \grad \delta^{1+\upepsilon}(x)\right| k(y)\D{y}\\
&=\int_{|y|< \frac{\delta(x)}{2}}+\int_{\frac{\delta(x)}{2}\leq |y|\leq 1 }+ \int_{|y|> 1}:=I_1+I_2+I_3 \, .
\end{aligned}
\end{equation*}
Since $|y| \leq \frac{\delta(x)}{2}$ and $\delta(x)<1,$ we have the following estimate on $I_1.$
\begin{align*}
\left|\delta^{1+\upepsilon}(x+y)-\delta^{1+\upepsilon}(x) -\Ind_{\sB_1}(y) y \cdot \grad \delta^{1+\upepsilon}(x)\right| &\leq ||\delta^{1+\upepsilon}||_{C^2(\sB_{\frac{\delta(x)}{2}}(x))}|y|^2 \\
\leq 4C \frac{||\delta||_{C^2(\overline{\Omega})}}{\delta(x)^{1-\upepsilon}}|y|^2 &\leq 4C \frac{||\delta||_{C^2(\overline{\Omega})} \delta(x)^{2-\alpha}}{\delta(x)^{1-\upepsilon}}|y|^{\alpha}.
\end{align*}
This implies 
\begin{align}\label{equ3.11}
I_1 \leq 4C ||\delta||_{C^2(\overline{\Omega})}\delta(x)^{1+\upepsilon-\alpha} \int_{\Rd} |y|^{\alpha}k(y)\D{y} \leq 4C_0C||\delta||_{C^2(\overline{\Omega})}\delta(x)^{1+\upepsilon-\alpha}.
\end{align}
Again for $I_2$ we have 
\begin{align*}
I_2 \leq C \int_{\frac{\delta(x)}{2} \leq |y| \leq 1}|y|k(y)\D{y} &\leq \left(\frac{C \delta(x)}{2} \right)^{1-\alpha}\int_{\frac{\delta(x)}{2} \leq |y| \leq 1}|y|^{\alpha}k(y)\D{y} \nonumber \\
&\leq \left(\frac{C \delta(x)}{2} \right)^{1-\alpha} \int_{\Rd} \left(1\wedge|y|^{\alpha}\right)k(y)\D{y}.
\end{align*}
Finally,
\begin{align}\label{equ3.13}
I_3=\int_{|y| >1} |\delta^{1+\upepsilon}(x+y)-\delta^{1+\upepsilon}(x)|k(y)\D{y}
\leq 2 (\diam \Omega)^{1+\upepsilon}
\int_{\Rd} (1 \wedge |y|^{\alpha}) k(y)\D{y}.
\end{align}
Combining \eqref{equ3.11}-\eqref{equ3.13} we obtain \eqref{delta estimate}, completing the proof.
\end{proof}
Next we obtain a suitable lower bound for minimal Pucci operator $\cP^-$ applied on $\delta^{1+\upepsilon}.$
\begin{lemma}\label{Lem3.3}
Let $\Omega$ be a bounded $C^2$ domain in $\Rd$, 
then for any $0<\upepsilon<1,$ we have the following estimate
\begin{align*}
\mathcal{P}^{-}\left(\delta^{1+\upepsilon}\right) \geq C_1 \cdot \upepsilon \delta^{\upepsilon-1}-C_2\,\,\,\,\text{in}\,\,\,\, \Omega ,
\end{align*}
where $C_1 , C_2 $ depends only on $d,\Omega , \lambda, \Lambda $.
\end{lemma}
\begin{proof}
Since $\delta\in C^2(\overline{\Omega}),$ for any $x\in \Omega,$ we can classically compute the following.
\begin{align*}
\frac{\partial ^2}{\partial x_i \partial x_j} \delta^{1+\upepsilon}(x)
=(1+\upepsilon)\Big[\delta^{\upepsilon}(x)\frac{\partial ^2}{\partial x_i \partial x_j} \delta(x)+ \upepsilon \delta^{\upepsilon -1}(x) \frac{\partial \delta(x)}{\partial x_i}\cdot\frac{\partial \delta(x)}{\partial x_j}\Big]:=A+B
\end{align*}
where $A, B$ are two $d \times d$ matrices given by $$A:= (a_{i,j})_{1\leq i,j \leq d}=(1+\upepsilon)\delta^{\upepsilon}(x)\frac{\partial ^2}{\partial x_i \partial x_j} \delta(x)$$ and $$B:= (b_{i,j})_{1\leq i,j \leq d}=(1+\upepsilon)\upepsilon \delta^{\upepsilon -1}(x) \frac{\partial \delta(x)}{\partial x_i}\cdot\frac{\partial \delta(x)}{\partial x_j}.$$
Note that $B$ is a positive definite matrix and $||A|| \leq d^2 (1+\upepsilon) (\diam \Omega)^{\upepsilon}||\delta||_{C^2(\overline{\Omega})}.$ Therefore we have
\begin{align*}
\cP^{-}\delta^{1+\upepsilon}(x) &\df \inf \left\lbrace \trace \left( N D^2\delta^{1+\upepsilon}(x)\right) , N \in  M^d ,  \lambda \Id \leq N \leq \Lambda \Id  \right\rbrace\\
&\df \inf \left\lbrace \trace \left( N (A+B)\right) , N \in  M^d ,  \lambda \Id \leq N \leq \Lambda \Id  \right\rbrace\\
&\geq \inf \left\lbrace \trace \left( N B\right) , N \in  M^d ,  \lambda \Id \leq N \leq \Lambda \Id  \right\rbrace+\inf \left\lbrace \trace \left( N A\right) , N \in  M^d ,  \lambda \Id \leq N \leq \Lambda \Id  \right\rbrace\\
&\geq \upepsilon(1+\epsilon) \delta^{\upepsilon-1}(x) \lambda |D\delta(x)|^2-d^2 \Lambda (1+\upepsilon) (\diam \Omega)^{\upepsilon}||\delta||_{C^2(\overline{\Omega})}\\
&\geq C_1 \cdot \upepsilon \delta^{\upepsilon-1}(x)-C_2.
\end{align*} This completes the proof.
\end{proof}
Next, we obtain an estimate on $\cL \delta$ in $\Omega.$
\begin{lemma}\label{Lem3.4}
Let $\Omega$ be a bounded $C^{2}$ domain in $\Rd$. Then we have the following estimate
\begin{align}\label{equ4.10}
   |\cL\delta|\leq C(1+\Ind_{(1,2)} \delta^{1-\alpha}) \,\,\, in \,\,\, \Omega,
\end{align}
where constant $C$ depends only on $d, \Omega, \lambda, \Lambda$ and $ \int_{\Rd} (1\wedge|y|^\alpha) k(y) \D{y}.$
\end{lemma}
\begin{proof}
First of all,  for all $x \in \Omega$ we have
\begin{align}\label{EL2.6C}
|\cL\delta(x)|  \leq \sup_{\theta,\nu} | \trace (a_{\theta \nu}  (x) D^2 \delta(x))|+ \sup_{\theta,\nu}|\cI_{\theta\nu} \delta(x)|
 \leq \kappa + \sup_{\theta,\nu} |\cI_{\theta\nu} \delta(x)|,
\end{align}
for some constant $\kappa$, depending on $\Omega$ and uniform bound of $a_{\theta \nu }$. 
For $\alpha \in (0,1],$ \eqref{equ4.10} \textcolor{black}{it} follows from the same arguments of \cref{Lem3.2}. For $\alpha\in (1,2),$ it is enough to obtain the estimate \eqref{equ4.10} for all $x \in \Omega$ such that $\delta(x) < 1$.
We follow the similar calculation as in \cref{Lem3.2} and get
\begin{align*}
\vert \cI_{\theta \nu} \delta(x) \vert &\leq\int_{\Rd} \vert \delta(x+y)-\delta(x) -\Ind_{B_1}(y) y \cdot \grad \delta(x)  \vert k(y)\D{y}
\\
&= \int_{|y|\leq\frac{\delta(x)}{2}} +\int_{\frac{\delta(x)}{2}<|y|<1}+\int_{|y|>1}  
\end{align*}
and
\begin{align*}
|I_{\theta \nu} \delta(x)| \leq \kappa_1 \int_{\Rd }(1\wedge |y|^\alpha) k(y) \D{y} \delta^{1-\alpha}(x)
\end{align*}
for some constant $\kappa_1$. Inserting these estimates in
\eqref{EL2.6C} we obtain
$$|\cL\delta(x)|\leq \kappa_2  \delta^{1-\alpha}(x)$$
for some constant $\kappa_2$ and 
\eqref{equ4.10} follows.
\end{proof}
Let us now define the sets near the boundary $\partial \Omega$ that we use for our oscillation estimates. We borrow the notation from \cite{RS14}.

\begin{definition}\label{D2.1}
Let $ \kappa\in(0, \frac{1}{16})$ be a fixed small constant, and let $ \kappa^{\prime} = 1/2 + 2\kappa$.
Given a point $x_0 \in \partial \Omega$  and $R>0$, we define
$$
\sD_R = \sD_R(x_0) = \sB_R(x_0) \cap \Omega,
$$
and 
$$
\sD^{+}_{\kappa^{\prime} R} = \sD^{+}_{\kappa^{\prime} R}(x_0) = \sB_{\kappa^{\prime} R} (x_0) \cap \left\lbrace  x \in \Omega : (x-x_0)  \cdot  {\rm n}(x_0) \geq 2\kappa R  \right\rbrace,
$$
where ${\rm n} (x_0)$ is the unit inward normal at $x_0$. For any bounded $C^{1,1}$-domain, we know that there exists $\rho>0$,  depending on $\Omega$,  such that the following inclusions hold for each $x_0 \in \partial \Omega$ and $R \leq \rho$:
\begin{equation}\label{E2.18}
\sB_{\kappa R}(y) \subset \sD_{R}(x_0) \qquad \text{ for all} \; y \in \sD^{+}_{\kappa^{\prime} R}(x_0),
\end{equation}
and 
\begin{equation}\label{E2.19}
\sB_{4\kappa R}(y^{\ast} + 4\kappa R {\rm n}(y^{\ast})) \subset \sD_R(x_0), \quad \text{and} \quad  
\sB_{\kappa R}(y^{\ast} + 4\kappa R {\rm n}(y^{\ast})) \subset \sD^{+}_{\kappa^{\prime} R}(x_0)
\end{equation}
for all $y\in \sD_{R/2}$, where $y^{\ast} \in \partial \Omega$ is the unique boundary point satisfying $|y-y^{\ast}| = \dist(y, \partial \Omega)$. Note that,  since $R\leq \rho$,  $y \in \sD_{R/2}$  is close enough to $\partial \Omega$ and hence the point 
$y^{\ast} + 4 \kappa R\, {\rm n}(y^{\ast})$ belongs to the line joining $y$ 
and $y^{\ast}$.
\end{definition}
\begin{remark}\label{R2.2}
In the remaining part of this section, we fix $\rho >0$ to be a small constant depending only on $\Omega$, so that \eqref{E2.18}-\eqref{E2.19} hold whenever $R \leq \rho$ and $x_0 \in \partial	\Omega$.
Also, every point on $\partial \Omega$ can be touched from both inside and outside $\Omega$ by balls of radius $\rho$. We
also fix $\upsigma \in (0, \upgamma)$ for small enough $\upgamma$ so that
for $0<r\leq \rho$  and $x_0 \in \partial \Omega$ we have 
$$
\sB_{\eta r}(x_0) \cap \Omega \subset \sB_{(1+\sigma)r} (z) \setminus \bar{\sB}_{r}(z)\quad \text{for}\quad \eta=\sigma/8,
$$
for any $x^{\prime} \in \partial \Omega \cap \sB_{\eta r} (x_0)$, where
 $\sB_r (z) $ is a ball contained in $\Rd \setminus \Omega$ that touches $\partial \Omega$ at point $x^{\prime}$. 
\end{remark}
In the following lemma, using \cref{Lem3.2} and \cref{Lem3.3} we construct supersolutions. 
\begin{lemma}\label{Lem3.5}
Let $\Omega$ be a bounded $C^2$ domain in $\Rd$ and $\alpha \in (1,2),$ then there exist $\rho_1 >0$ and a $C^2$ function  $\phi_1 $ satisfying 
$$
\begin{cases}
\cP^{+} \phi_1(x) + \cP^{+}_k \phi_{1}(x) \leq -C\delta^{- \frac{\alpha}{2}}(x) &\quad \text{in } \quad \Omega_{\rho_1}, \\
C^{-1} \delta(x) \leq \phi_1(x) \leq C \delta(x) &\quad \text{in} \quad \Omega, \\
\phi_1(x) = 0 &\quad \text{in} \quad \Rd \setminus \Omega ,
\end{cases}   
$$
where the constants $\rho_1$ and $C$ depend only on $d, \alpha, \Omega,\lambda, \Lambda $ and $\int_{\Rd} (1 \wedge |y|^{\alpha} ) k(y) \D{y}$.
\end{lemma}
\begin{proof}
Let $\upepsilon = \frac{2-\alpha}{2}$ and  $c = \frac{1}{(\diam \Omega)^2}$, and define 
$$
\phi_1(x) = \delta (x)- c \delta^{1+\upepsilon}(x).
$$
Since both $\delta$ and $\delta^{1+\upepsilon}$ are in $C^2(\Omega)$,  we have
$\cP^{+} \phi_1(x) \leq \cP^{+} \delta(x) -c \cP^{-} \delta^{1+\upepsilon}(x).$
Then by \cref{Lem3.3}  and $\sup_{\theta\nu} | \trace (a_{\theta \nu}(x) D^2 \delta(x))| \leq \tilde{C}$,  we get for all $x \in \Omega_{\rho}$
\begin{align*}
\cP^{+} \phi_1(x) &\leq \cP^{+} \delta(x) -c \cP^{-} \delta^{1+\epsilon}(x) 
\leq C - c (C_1\cdot \upepsilon \delta^{\upepsilon - 1} (x)).
\end{align*}
Similarly for all $x \in \Omega_{\rho}$,  using \cref{Lem3.2} and \cref{Lem3.4} we get
$$
\cP^{+}_k \phi_1 (x) \leq  | \cP^{+}_k  \delta (x)| + c | \cP^{-}_k \delta^{1+\upepsilon} (x) | \leq C_2 \delta^{1-\alpha}(x).
$$
Combining the above inequalities we have
\begin{align*}
\cP^{+} \phi_1 (x) + \cP_k^{+}\phi_1(x) &\leq C-cC_1 \upepsilon \delta^{\upepsilon-1}(x) + C_2 \delta^{1-\alpha}(x)\\
&\leq -\delta^{\upepsilon-1}(x) \left( \frac{C_1 (2-\alpha)}{2 (\diam \Omega)^2} - C \delta^{\frac{\alpha}{2}}(x) - C_2 \delta^{\frac{2-\alpha}{2}}(x) \right) ,
\end{align*}
for all $x \in \Omega_{\rho}$. Now choose $0<\rho_1\leq \rho < 1$ such that 
$$
\left( \frac{C_1 (2-\alpha)}{2 (\diam \Omega)^2} - C \rho_1^{\frac{\alpha}{2}} - C_2 \rho_1^{\frac{2-\alpha}{2}} \right) \geq \frac{C_1 (2-\alpha)}{4 (\diam \Omega)^2}.
$$
Thus for all $x \in \Omega_{\rho_1},$ we have 
$$
\cP^{+} \phi_1 (x) + \cP_k^{+}\phi_1(x) \leq  - \frac{C_1 (2-\alpha)}{4 (\diam \Omega)^2}  \delta^{-\frac{\alpha}{2}}(x). 
$$
Finally the construction of $\phi_1$  immediately gives us that 
$$
C^{-1} \delta(x) \leq \phi_1(x) \leq C \delta(x) \quad \text{in} \quad \Omega,
$$
and $\phi_1 =0$ in $\Omega^c.$ This completes the proof of the lemma.
\end{proof}
As mentioned \textcolor{black}{in the introduction}, the key step of proving \cref{T1.2} is to obtain the oscillation lemma \cref{P2.1}. For this we next prove two preparatory lemmas. In the first lemma we obtain a lower bound of $\inf_{\sD_{\frac{R}{2}}} \frac{u}{\delta}$ whereas the second lemma controls $\sup_{\sD^+_{\kappa^{\prime}R}} \frac{u}{\delta}$ by using that lower bound. 
\begin{lemma}\label{Lem3.7}
Let $\alpha\in(0,2)$ and  $\Omega$ be a bounded $C^{2}$ domain in $\Rd$. Also, let $u$ be such that $u \geq 0$ in $\Rd$,  and $ \abs{\cL u} \leq C_2(1+  \Ind_{(1,2)}(\alpha)\delta^{1-\alpha}) $ in $\sD_R$, for some constant $C_2$. If $\hat{\alpha}$ is given by
\begin{align*}
    \hat\alpha=\begin{cases}
    1\,\,\,\, &if \,\,\,\, \alpha \in (0,1],\\
    \frac{2-\alpha}{2}\,\,\,\, &if \,\,\,\, \alpha \in (1,2),
        \end{cases}
\end{align*}
then there exists a positive constant $C$ depending only on $d,  \Omega ,  \Lambda, \lambda, \alpha, \int_{\Rd} (1 \wedge |y|^{\alpha}) k(y)\D{y}$, such that
\begin{equation}\label{EL2.3A}
\inf_{\sD^{+}_{\kappa^{\prime}R}}  \frac{u}{\delta}  \leq C \left( \inf_{\sD_{\frac{R}{2}}} \frac{u}{\delta} + C_2 R^{\hat \alpha} \right)
\end{equation}
for all $R \leq \rho_0$,  where the constant $\rho_0$ depends only on 
$d,  \Omega,  \lambda , \Lambda, \alpha$ and $\int_{\Rd}( 1 \wedge |y|^\alpha ) k(y)\D{y}$.
\end{lemma}
\begin{proof}
 Suppose $R \leq \eta \rho$, where
$\rho$ is given by \cref{R2.2} and $\eta \leq 1$ be some constant that will be chosen later.
Define $m= \inf_{\sD^{+}_{\kappa^{'}R} } u/ \delta \geq 0$.
Let us first observe that by \eqref{E2.18} we have, 
\begin{equation}\label{EL2.3B}
u \geq m\delta \geq m\,(\kappa R)\quad  \text{in}\;\; \sD^{+}_{\kappa^{\prime}R}.
\end{equation}
Moreover by \eqref{E2.19}, for any $y \in \sD_{R/2},$ we have either $y \in \sD^+_{\kappa^{\prime}R}$ or $\delta(y) < 4\kappa R.$ If $y \in \sD^+_{\kappa^{\prime}R},$ then by the definition of $m$ we get $m \leq u(y)/ \delta(y).$

 Next we consider $\delta(y) < 4 \kappa R$.
Let $y^{\ast}$ be the nearest point to $y$ on $\partial \Omega,$ i.e, $\dist(y, \partial \Omega)=|y-y^*|$ and define
$ \tilde{y} = y^{\ast} + 4 \kappa R\, {\rm n} (y^{\ast})$.  Again by \eqref{E2.19},   we have $$\sB_{4\kappa R}(\tilde{y})  \subset \sD_{R} \,\,\,\,  \text{and} \,\,\,\, \sB_{\kappa R}(\tilde{y}) \subset \sD^{+}_{\kappa^{\prime}R} .$$ 
Denoting $r=\kappa R$ and using the subsolution constructed in \cref{L2.2}, define $\tilde{\phi}_{r}(x):=\frac{1}{\tilde{\kappa}}\phi_{r}(x-\tilde{y}).$
 We will consider two cases. \\
\noindent{\textbf{Case 1:}}  $\alpha \in (0,1]$.
Take $r^{\prime} = \frac{R}{\eta}$.  Since $r^{\prime} \leq \rho$,  points of $\partial\Omega$  can be touched by exterior ball of radius $r^\prime$.  In particular,  for $y^{\ast}\in\partial\Omega$, 
we can find a point $z\in \Omega^c$ such that $\bar{\sB}_{r^{\prime}}(z) \subset \Omega^c$ touches
$\partial \Omega$ at $y^{\ast}$. Now from  \cite[Lemma~5.4]{M19}
there exists a bounded, Lipschitz continuous function $\varphi_{r^{\prime}}$, with Lipschitz constant  $\frac{1}{r^{\prime}}$, that satisfies
$$
\begin{cases}
\varphi_{r^\prime} = 0, & \quad \text{in} \quad \bar{\sB}_{r^\prime} , 
\\
\varphi_{r^\prime} > 0, & \quad \text{in} \quad \bar{\sB}_{r^\prime}^c ,
 \\
\cP^+ \varphi_{r^\prime} + \cP^+_k \varphi_{r^\prime} \leq - \frac{1}{(r^{\prime})^2} , & \quad \text{in} \quad \sB_{(1+\sigma)r^\prime}\setminus \bar{\sB}_{r^\prime},
\end{cases}
$$
for some constant $\sigma$, independent of $r^{\prime}$. Without any loss of any
generality we may assume $\sigma\leq\upgamma$ (see \cref{R2.2}).
 Then setting $\eta =\frac{\sigma}{8}$ and using \cref{R2.2}, we have $$\sD_{R} \subset \sB_{(1+\sigma)r^{\prime}}(z) \setminus  \overline{\sB}_{r^\prime}(z)$$ and by \eqref{E2.19} we have 
 \[\sB_{4r}(\tilde{y})\setminus \overline{\sB}_{r}(\tilde{y}) \subset \sD_{R}\subset \sB_{(1+\sigma)r^{\prime}}(z)\setminus \overline{\sB}_{r^{\prime}}(z) .\]
We show that $v(x)=m \tilde{\phi}_{r}(x)-C_2 (r^{\prime})^2 \varphi_{r^{\prime}}(x-z)$ is an appropriate subsolution. Since both $\tilde{\phi}_{r}$ and $\varphi_{r^{\prime}}$ are $C^2$ functions in $\sB_{4r}(\tilde{y})\setminus \bar{\sB}_{r}(\tilde{y}),$ we conclude that $v$ is $C^2$ function in $\sB_{4r}(\tilde{y})\setminus \bar{\sB}_{r}(\tilde{y}).$ For $x \in \sB_{4r}(\tilde{y})\setminus \bar{\sB}_{r}(\tilde{y}),$
\[\cP^-v(x)+\cP^-_k v(x) \geq m \left[\cP^-\tilde{\phi}_{r}(x)+\cP^-_k \tilde{\phi}_{r}(x)\right]-C_2 (r^{\prime})^2\left[\cP^+ \varphi_{r^{\prime}}(x-z)+\cP^+_k \varphi_{r^{\prime}}(x-z)\right] \geq C_2.\] Therefore by \cref{rem2.1} we have 
\[\cP^+ ( v- u ) +\cP^+_k ( v - u ) \geq 0 \,\,\,\, \text{in}\,\,\,\, \sB_{4r}(\tilde{y}) \setminus \bar{\sB}_{r} (\tilde{y}).\] Furthermore, using \eqref{EL2.3B} and $u\geq 0$ in $\Rd$ we obtain $u(x) \geq m \tilde{\phi}_{r}(x)- C_2 (r^{\prime})^2 \varphi_{r^{\prime}}(x-z)$ in $\left(\sB_{4r}(\tilde{y}) \setminus \bar{\sB}_{r} (\tilde{y})\right)^c.$ Hence an application of maximum principle (cf. \cref{A3.1}) gives $u \geq v$ in $\Rd.$ Now for  $y \in \sD_{R/2}$,  using the Lipschitz continuity of $\varphi_{r^{\prime}}$ we get
\begin{align*}\label{eqq4.14}
    m \tilde{\phi}_{r}(y) \leq u(y)+ C_2 (r^{\prime})^2 \left[ \varphi_{r^{\prime}}(y-z)-\varphi_{r^{\prime}}(y^*-z)\right] \leq u(y)+C_2 {r^{\prime}} \cdot \delta(y)
\end{align*}
and as $y$ lies on the line segment joining $y^*$ to $\tilde{y}$ we get
\[\frac{u(y)}{\delta(y)} + C_2 r^{\prime} \geq \frac{m}{(\tilde{\kappa})^2}.\]  This gives
\[\inf_{\sD^+_{\kappa^{\prime}R}}\frac{u}{\delta} \leq C \left(\inf_{\sD_{R/2}}\frac{u}{\delta}+C_2 \frac{R}{\eta}\right)\] and finally choosing $\rho_0=\eta \rho$ we have \eqref{EL2.3A}.

\noindent{\textbf{Case 2:}} $\alpha \in (1,2)$.  Let $\rho_1$ as in \cref{Lem3.5} and consider $R \leq \rho_1 <1$.  Here we aim to construct an appropriate subsolution using $\tilde{\phi}_r(x) $  and supersolution constructed in \cref{Lem3.5}. Since $\delta(x) \leq 1$ in $\sD_R,$ we have $ \abs{\cL u (x)} \leq C_2(1+  \delta^{1-\alpha}(x))\leq 2C_2 \delta^{1-\alpha}(x) $ in $\sD_R$. Also by \cref{Lem3.5}, we have a bounded function $\phi_1$ which is $C^2$ in $\Omega_{\rho_1} \supset \sD_R$ and satisfies
\begin{align*}
\cP^{+} \phi_1(x) + \cP^{+}_k \phi_{1}(x) \leq -C\delta^{- \frac{\alpha}{2}}(x) = -C \frac{1}{\delta^{ \frac{2-\alpha}{2}} (x)} \delta^{1-\alpha}(x)
\leq  \frac{-C}{R^{\hat \alpha}} \delta^{1-\alpha}(x),
\end{align*}
for all $x \in \sD_R$.  Now we define the subsolutions as
\begin{equation*}
        v(x)=m \tilde{\phi}_r(x) -  \upmu \, R^{\hat\alpha} \phi_1(x), 
\end{equation*}
where the constant $\upmu$ is chosen suitably so that $ \cP^- v(x) + \cP^-_k v(x) \geq  2C_2 \delta^{1-\alpha}(x)$ in $\sB_{4r}(\tilde{y}) \setminus \bar{\sB}_r (\tilde{y})$  (i.e. $\upmu = \frac{2C_2}{C}$).  Also $u\geq v$ in $(\sB_{4r}(\tilde{y}) \setminus \bar{\sB}_r (\tilde{y}))^c$. 
 Using the same calculation as previous case for $v-u$ and  maximum principle \cref{A3.1} we derive that $ u \geq v$ in $\Rd$. Again, repeating the arguments of \textbf{Case 1} we get
 $$
\inf_{\sD^{+}_{\kappa^{\prime}R}}  \frac{u}{\delta} \leq C \left( \inf_{\sD_{\frac{R}{2}}} \frac{u}{\delta} + 2C_2 R^{\hat{\alpha}} \right)    \,.
$$
 Choosing $\rho_0 = \eta \rho \wedge \rho_1$ completes the proof.
\end{proof}
\begin{lemma}\label{L2.5}
Let $\alpha^\prime = 1 \wedge (2-\alpha)$ and  $\Omega$ be a bounded $C^{2}$ domain in $\Rd$. Also, let 
$u$ be a bounded continuous function such that $u \geq 0$ and $u \leq M_0 \delta(x)$  in $\Rd$,  and $ \abs{\cL u} \leq C_2(1+  \Ind_{(1,2)}(\alpha)\delta^{1-\alpha}) $ in $\sD_R$, for some constant $C_2$.
 Then, there exists a positive constant $C$,  depending only on $d, \lambda, \Lambda,  \Omega$ and $\int_{\Rd}(1 \wedge |y|^\alpha){k}(y)\D{y}$, such that
\begin{equation}\label{EL2.5A}
\sup_{\sD^{+}_{\kappa^{\prime}R}}  \frac{u}{\delta}  \leq C \left( \inf_{\sD^{+}_{\kappa^{\prime}R}} \frac{u}{\delta} + (M_0\vee C_2) R^{\alpha^\prime} \right)
\end{equation}
for all $R \leq \rho$,  where constant $\rho$ is given by \cref{R2.2}.
\end{lemma}
\begin{proof}

 We will use the weak  Harnack inequality proved in \cref{T5.1} to show  \eqref{EL2.5A}.  Let $R \leq \rho$. Then for each $y \in \sD^{+}_{\kappa^{\prime}R}$,  we have $\sB_{\kappa R}(y) \subset \sD_R $.  Hence we have $ \abs{\cL u} \leq C_2(1+  \Ind_{(1,2)}(\alpha)\delta^{1-\alpha}(x)) $ in $\sB_{\kappa R}(y)$.  Without loss of generality, we may assume $y=0$. Let $s=\kappa R$  and define  $v(x) = u(sx)$ for all $x \in \Rd$.  Then,
it can be easily seen that 
$$
s^2 \cL [sx,u] =   \cL^s[x,v]  \df  \sup_{\theta \in \Theta} \inf_{\nu \in \Gamma}\left\{\trace a_{\theta \nu}(sx)D^2 v(x) + \cI^s_{\theta \nu}[x,v]\right\}  \qquad \text{for all } \; x \in \sB_2.
$$
This gives 
\begin{align*}
\left|\cL^s[x,v] \right| &\leq  C_2 s^2 ( 1 +  \Ind_{(1,2)}(\alpha)\delta^{1-\alpha}(sx))\\
	&\leq C_2  \left( s^2 +  \Ind_{(1,2)}(\alpha) s^2 \left(\kappa R\right)^{1-\alpha} \right)\\
	&\leq C_2 s^{1+\alpha^{\prime}} \,  ,
\end{align*}
 in $\sB_2$  where $\alpha^\prime = 1 \wedge (2-\alpha)$.  In second line,  we used that for each $x\in \sB_{\kappa R},$ $|sx| < \kappa R$ and hence $\delta(sx) > \frac{\kappa R}{2}=\frac{s}{2}$.   From $u \leq M_0 \delta(x)$ we have $v(y) \leq M_0 \diam \Omega$  and $v(y) \leq M_0 s (1+ |y|) $ in whole $\Rd$.
Hence by \cref{Col3.1}, we obtain 
$$
\sup_{\sB_{\frac{1}{4}}} v \leq C \left(\inf_{\sB_{\frac{1}{4}}} v  + (M_0\vee C_2) s^{1+  \alpha^{\prime}}\right) ,
$$
where constant $C$ does not depend on $s,M_0,C_2$.  This of course, implies
$$
\sup_{\sB_{\frac{\kappa R}{64}}(y) } u \leq C \left( \inf_{\sB_{\frac{\kappa R}{64}}(y)} u +(M_0\vee C_2)  R^{1+\alpha^\prime}\right),
$$
for all $y \in \sD^{+}_{\kappa^{\prime}R}$.  Now cover $\sD^{+}_{\kappa^{\prime}R}$ by a finite number of balls $\sB_{\kappa R/64} (y_i)$, independent of $R$,  to obtain
$$
\sup_{\sD^{+}_{\kappa^{\prime}R}} u \leq C \left(\inf_{\sD^{+}_{\kappa^{\prime}R} } u + (M_0\vee C_2) R^{1+\alpha^\prime}\right).
$$
Then \eqref{EL2.5A} follows since $ \kappa R/2 \leq \delta \leq 3 \kappa R/2$ in $\sD^{+}_{\kappa^{\prime}R}$.
\end{proof}
Now we are ready to prove the oscillation lemma.
\begin{proposition}\label{P2.1}
Let $u$ be a bounded continuous function such that $|\cL u| \leq K$ in
$\Omega$, for some constant $K$, and $u=0$ in $\Omega^c$. Given any $x_0 \in \partial \Omega$, let $\sD_R$ be as in the \cref{D2.1}. Then for some $\uptau \in (0, \hat{\alpha})$ there exists $C$, dependent on 
$\Omega, d, \lambda, \Lambda, \alpha$ and $\int_{\Rd}(1\wedge|y|^\alpha)k(y)\D{y}$ but not on $x_0$, such that
\begin{equation}\label{EP2.1A}
    \sup_{\sD_R} \frac{u}{\delta} -\inf _{\sD_R} \frac{u}{\delta} \leq CK R^{\uptau}
\end{equation}
for all $R\leq \rho_0$, where $\rho_0>0$ is a constant depending only on $\Omega, d, \lambda, \Lambda, \alpha$ and $\int_{\Rd}(1\wedge|y|^\alpha)k(y)\D{y}$.
\end{proposition}

\begin{proof}
For the proof, we follow a standard method similar to \cite{RS14},
with the help of \cref{Lem3.4,L2.5,Lem3.7}.
Fix $x_0\in \partial \Omega$ and consider $\rho_0>0$ to be chosen later. Without loss of generality, we assume $x_0=0$.
In view of \eqref{EL2.1B}, we only consider the case $K>0$.
Taking $u/K$ instead of $u$, we can further assume that $K=1$, that is, $|\cL u| \leq 1$ in $\Omega$. From \cref{T1.1} we note that
$||u||_{C^{0,1}(\Rd)}\leq C_1$. We split the proof into two cases.

\medskip
\noindent{\textbf{Case 1:}} For $\alpha\in (0, 1]$, $\cI_{\theta\nu}u$ is classically defined and $|\cI_{\theta \nu}u|\leq \tilde{C}$ in $\Omega$ for all $\theta$ and $\nu.$ Consequently, one can combine the nonlocal term with the right hand side and only deal with local nonlinear operator $ \tilde \cL [x,u]:= \sup_{\theta \in \Theta} \inf_{\nu \in \Gamma} \left\{\trace a_{\theta \nu}(x) D^2 u(x)\right\}.$ In this case the proof is simpler and can be done following the same method as for the local case. However, the method we use below would also work with an appropriate modification.

\medskip
\noindent{\textbf{Case 2:}}  Now we deal with the case $\alpha\in (1,2)$. We show that there exist $\sK>0,$ $\rho_1\in (0, \rho_0)$ and $\uptau \in (0,1)$, dependent only on $\Omega, d,\lambda, \Lambda, \alpha$ and $\int_{\Rd}(1\wedge|y|^\alpha)k(y)\D{y}$, and monotone sequences $\{M_k\}$ and $\{m_k\}$ such that, for all $k\geq 0,$
\begin{align}\label{EP2.1B}
    M_k-m_k=\frac{1}{4^{k\uptau}}, \,\,\,\,\, -1\leq m_k \leq m_{k+1}<M_{k+1}\leq M_k \leq 1,
\end{align}
and
\begin{align}\label{EP2.1C}
    m_k \leq \sK^{-1} \frac{u}{\delta} \leq M_k \quad \text{in} \quad \sD_{R_k},  \quad \text{where}\quad R_k=\frac{\rho_1}{4^k}. 
\end{align}
Note that \eqref{EP2.1C} is equivalent to the following
\begin{align}\label{EP2.1D}
    m_k \delta \leq \sK^{-1} u \leq M_k \delta, \quad \text{in} \quad \sB_{R_k},  \quad \text{where}\quad R_k=\frac{\rho_1}{4^k}.
\end{align}
Next, we construct monotone sequences $\{M_k\}$ and $\{m_k\}$ by induction.

The existence of $M_0$ and $m_0$ such that \eqref{EP2.1B} and \eqref{EP2.1D} hold for $k=0$ is guaranteed by \cref{Barrier1}.
Assume that we have the sequences up to $M_k$ and $m_k$. We want to show the existence of $M_{k+1}$ and $m_{k+1}$ such that \eqref{EP2.1B}-\eqref{EP2.1D} hold. We set
\begin{align*}
    u_k=\frac{1}{\sK} u -m_k \delta.
\end{align*}
Note that to apply \cref{L2.5} we need $u_k$ to be nonnegative in $\Rd$.
Therefore we shall work with $u^+_k$, the positive part of $u_k$. Let $u_k=u^+_k-u^-_k$ and by the induction hypothesis,
\begin{align}\label{EP2.1F}
    u^+_k=u_k\quad \text{and} \quad  u^-_k=0\quad \text{in} \quad \sB_{R_k}.
 \end{align}
We need to find a lower bound on $u_k$. Since $u_k\geq 0$ in $\sB_{R_k}$ and $u_k$ is Lipschitz in $\Rd,$
we get for $x\in \sB^c_{R_k}$ that
\begin{align}\label{EP2.1G}
u_k(x)=u_k(R_k x_{\rm u}) + u_k(x)-u_k(R_k x_{\rm u})
\geq -C_L |x-R_k x_{\rm u}|,
\end{align}
where $z_{\rm u}=\frac{1}{|z|}z$ for $z\neq 0$ and $C_L$ denotes a
Lipschitz constant of $u_k$ which can be chosen independent of $k$.
Using \cref{Barrier1} we
 also have $|u_k|\leq  \sK^{-1}+\diam(\Omega)=C_1$ for all
 $x\in\Rd$.
Thus using \eqref{EP2.1F} and \eqref{EP2.1G} we calculate
$\cL [x,u^-_k]$ in $\sD_{\frac{R_k}{2}}$.
Let $x \in \sD_{R_k/2}$. By \eqref{EP2.1F}, $D^2 u^-_k(x)=0$. 
Then
\begin{align}\label{E4.23}
 0\leq \cI_{\theta\nu} [x,u^-_k]&= \int _{x+y \not \in \sB_{R_k}} u^-_k(x+y) N_{\theta\nu}(x,y)\D{y}
 \nonumber \\
&\leq\int _{ \left\lbrace |y|\geq \frac{R_k}{2}, x+y\neq 0 \right\rbrace } u^-_k(x+y)k(y)\D{y} \nonumber
 \\
&\leq C_L\int_{\left\lbrace \frac{R_k}{2}\leq |y| \leq 1,\; x+y\neq 0 \right\rbrace } \Big|(x+y)-R_k(x+y)_{\rm u}\Big| k(y)\D{y}
+ C_1\int_{|y|\geq 1} k(y)\D{y} \nonumber
\\
 &\leq C_L  \int_{\frac{R_k}{2}\leq |y| \leq 1}\left(|x|+R_k\right)k(y)\, \D{y} +
 C_L  \int_{\frac{R_k}{2}\leq |y| \leq 1} |y| k(y)\, \D{y} + C_1 \int_{\Rd} (1 \wedge |y|^{\alpha}) k(y) \,\D{y} \nonumber
 \\
  &\leq \kappa_3 \left[ \int_{\Rd} (1\wedge |y|^{\alpha}) k(y) \,\D{y} \right] \left(R_k^{1-\alpha} +1\right) \nonumber
  \\
  &\leq \kappa_4 R_k^{1-\alpha},
    \end{align}
for some constants $\kappa_3, \kappa_4$, independent of 
$k$.

Now we write $u^+_k=\sK^{-1} u -m_k \delta + u^-_k$. Since $\delta$ is $C^2$ and $u^-_k=0$ in $\sD_{\frac{R_k}{2}},$ first note that
\begin{align*}
    \cL u^{+}_k &\leq \sK^{-1} - (\cP^{-}+\cP^{-}_k)(m_k \delta) + (\cP^{+}+\cP^{+}_k)(u_k^{-}) , \\
    \cL u^{+}_k &\geq - \sK^{-1} - (\cP^{+}+\cP^{+}_k)(m_k \delta) + (\cP^{-}+\cP^{-}_k)(u_k^{-}).
\end{align*}
Using \cref{Lem3.4} and \eqref{E4.23} in the above estimate, we have
\begin{align}\label{EP2.1H}
    |\cL u^+_k| \leq \sK^{-1} + m_k C \delta^{1-\alpha} + \kappa_4 (R_k)^{1-\alpha} \,\,\, \text{in}\,\,\, \sD_{\frac{R_k}{2}}.
\end{align}
Since $\rho_1\geq R_k \geq \delta$ in $\sD_{R_k}$, for $\alpha > 1,$ we have $R^{1-\alpha}_k \leq \delta^{1-\alpha}$, and hence, from \eqref{EP2.1H}, we have
\[|\cL u^+_k| \leq \Big[\sK^{-1} [(\rho_1)]^{\alpha-1} + C + \kappa_4\Big] \delta^{1-\alpha}(x): =\kappa_5 \delta^{1-\alpha}(x) \quad \text{in}\quad \sD_{R_k/2}.\]
Now we are in a position to apply \cref{Lem3.7,L2.5}.
Recalling that
 \[u^+_k=u_k=\sK^{-1} u -m_k \delta \quad \text{in}\quad \sD_{R_k},\] 
and using \cref{Barrier1} we
 also have $|u^+_k|\leq |u_k|\leq  (\sK^{-1}+1) \delta(x) =C_1 \delta(x)$ for all
 $x\in\Rd$. 
 We get from \cref{Lem3.7,L2.5} that
 \begin{equation}\label{EP2.1I}
 \begin{aligned}
     \sup_{\sD^+_{\kappa^{\prime} R_k/2}} \Big(\sK^{-1} \frac{u}{\delta}-m_k\Big) &\leq C\Big(\inf_{\sD^+_{\kappa^{\prime} R_k/2}}\Big(\sK^{-1} \frac{u}{\delta}-m_k\Big) + (\kappa_5 \vee C_1) R^{\hat\alpha}_k\Big)\\
     &\leq C\Big(\inf_{\sD_{ R_k/4}}\Big(\sK^{-1} \frac{u}{\delta}-m_k\Big) + (\kappa_5 \vee C_1) R^{\hat\alpha}_k\Big).
 \end{aligned}
 \end{equation}
Repeating a similar argument for the function $\tilde{u}_k=M_k\delta- \sK^{-1} u$, we find
 \begin{align}\label{EP2.1J}
     \sup_{\sD^+_{\kappa^{\prime} R_k/2}} \Big(M_k-\sK^{-1} \frac{u}{\delta}\Big)\leq C\Big(\inf_{\sD_{ R_k/4}}\Big(M_k-\sK^{-1} \frac{u}{\delta}\Big) + (\kappa_5 \vee C_1) R^{\hat\alpha}_k\Big).
 \end{align}
Combining \eqref{EP2.1I} and \eqref{EP2.1J} we obtain the following.
     \begin{align}\label{EP2.1K}
         M_k-m_k &\leq C\Big(\inf_{\sD^+_{ R_k/4}}\Big(M_k-\sK^{-1} \frac{u}{\delta}\Big) +\inf_{\sD^+_{ R_k/4}}
         \Big(\sK^{-1} \frac{u}{\delta}-m_k\Big)+ (\kappa_5 \vee C_1) R^{\hat\alpha}_k\Big)\nonumber
         \\
         &=C\Big(\inf_{\sD_{ R_{k+1}}}\sK^{-1} \frac{u}{\delta}-\sup_{\sD_{ R_{k+1}}}\sK^{-1} \frac{u}{\delta} + M_k-m_k + (\kappa_5 \vee C_1) R^{\hat\alpha}_k\Big).
     \end{align}
Putting $M_k-m_k=\frac{1}{4^{\uptau k}}$ in \eqref{EP2.1K}, we have
 \begin{align}\label{EP2.1L}
         \sup_{\sD_{ R_{k+1}}}\sK^{-1} \frac{u}{\delta}-\inf_{\sD_{ R_{k+1}}}\sK^{-1} \frac{u}{\delta} &\leq \Big(\frac{C-1}{C}\frac{1}{4^{\uptau k}}+(\kappa_5 \vee C_1) R^{\hat\alpha}_k\Big)\nonumber
         \\
 &=\frac{1}{4^{\uptau k}} \Big(\frac{C-1}{C}+(\kappa_5 \vee C_1) R^{\hat\alpha}_k 4^{\uptau k}\Big).
  \end{align}
Since $R_k=\frac{\rho_1}{4^k}$ for $\rho_1\in (0, \rho_0)$,
we can choose $\rho_0$ and $\uptau$ small so that
\[
\Big(\frac{C-1}{C}+(\kappa_5 \vee C_1) R^{\hat\alpha}_k 4^{\uptau k}\Big) \leq \frac{1}{4^\uptau}.
\]
Putting in  \eqref{EP2.1L} we obtain 
 \[
 \sup_{\sD_{ R_{k+1}}}\sK^{-1} \frac{u}{\delta}-\inf_{\sD_{ R_{k+1}}}\sK^{-1} \frac{u}{\delta} \leq \frac{1}{4^{\uptau(k+1)}}.
 \]
Thus we find $m_{k+1}$ and $M_{k+1}$ such that \eqref{EP2.1B} and \eqref{EP2.1C} hold. It is easy to prove \eqref{EP2.1A} from
\eqref{EP2.1B}-\eqref{EP2.1C}.
\end{proof}
Next we establish \textcolor{black}{the} H\"{o}lder regularity of $u/\delta$ up to the boundary, that is \cref{T1.2}.
\begin{proof}[Proof of \cref{T1.2}]
Replacing $u$ by $\frac{u}{CK}$ we may assume that $|\cL u|\leq 1$
in $\Omega$. Let $v=u/\delta$. From \cref{Barrier1} we then have
\begin{equation*}\label{ET2.2A0}
\norm{v}_{L^\infty(\Omega)}\leq C,
\end{equation*}
for some constant $C$ and from \cref{T1.1} we have
\begin{equation}\label{ET2.2B}
\norm{u}_{C^{0,1}(\Rd)}\leq C.
\end{equation}
Also from \cref{P2.1} for each $x_0 \in \partial \Omega$ and for all $r>0$ we have
\begin{equation}\label{ET2.2D}
\sup_{\sD_r(x_0)} v - \inf_{\sD_r(x_0)} v \leq C r^{\uptau}.
\end{equation}
where $\sD_r(x_0) = \sB_r(x_0) \cap \Omega$ as before. To complete the
proof we shall show that 
\begin{equation}\label{ET2.2E}
\sup_{x, y\in\Omega, x\neq y}\frac{|v(x)-v(y)|}{|x-y|^\upkappa}\leq C,
\end{equation}
for some $\upkappa>0$. Let $r=|x-y|$ and there exists $x_0, y_0 \in \partial \Omega$ such that $\delta(x)=|x-x_0|$ and $\delta(y)=|y-y_0|.$ If $r\geq \frac{1}{8},$ then
\begin{align*}
    \frac{|v(x)-v(y)|}{|x-y|^{\upkappa}}\leq 2\cdot 8^{\upkappa}||v||_{L^{\infty}(\Omega)}.
\end{align*}
If $r<\frac{1}{8}$ and $r \geq \frac{1}{8}(\delta(x)\vee \delta(y))^p$ for some $p>2$ then clearly $x, y\in \sB_{(8r)^{1/p}}(x_0)$. Now using \eqref{ET2.2D} we obtain
\begin{align*}
    |v(x)-v(y)| \leq \sup_{\sD_{(8r)^{1/p}}(x_0)} v - \inf_{\sD_{ (8r)^{1/p}}(x_0)} v \leq 8 C  r^{\uptau/p}.
\end{align*}
If $r<\frac{1}{8}$ and $r < \frac{1}{8}(\delta(x)\vee \delta(y))^p,$ then $r<\frac{1}{8}(\delta(x)\vee \delta(y))$. This implies $y\in \sB_{\frac{1}{8}(\delta(x)\vee \delta(y))}(x)$ and $x\in \sB_{\frac{1}{8}(\delta(x)\vee \delta(y))}(y).$ Without loss of any generality assume $\delta(x)\geq \delta(y)$ and $y\in \sB_{\frac{\delta(x)}{8}}(x).$ Notice that this also gives us that $\delta(y)\geq \frac{\delta(x)}{8} .$  Using \eqref{ET2.2B} and the Lipschitz continuity of $\delta,$ we get
\begin{align*}
   |v(x)-v(y)|=\left|\frac{u(x)}{\delta(x)}-\frac{u(y)}{\delta(y)}\right|\leq  C \diam \Omega \cdot \frac{ r }{\delta(x)\cdot\delta(y)} \; .
\end{align*}
Also we have $\frac{1}{8} r^{2/p}  <\delta(x)\cdot\delta(y).$ This implies
\begin{align*}
    |v(x)-v(y)|\leq C \diam \Omega \cdot \frac{ r}{\delta(x)\cdot\delta(y)}< 8 C \diam \Omega \cdot r^{1-2/p}.
\end{align*}
Therefore choosing $\upkappa=(1-\frac{2}{p})\wedge \frac{\uptau}{p}$ we conclude \eqref{ET2.2E}. This completes the proof.
\end{proof}
\begin{remark}
We want to point out that all the analysis presented in this section remains valid for $\delta \in C^2({\Omega_{\rho}}),$ with the involved constants depending additionally on $\rho$.  However, for the sake of simplicity, we consider a $C^2$-extension of $\delta$ inside the domain $\Omega$. This $C^2$-extension of the distance function is crucially used in the following section.
\end{remark}
\section{Global H\"{o}lder regularity of the gradient}\label{section 5}
In this section we prove the H\"{o}lder regularity of $Du$ up to the boundary.
First, let us recall \[\cL [x,u]= \sup_{\theta \in \Theta} \inf_{\nu \in \gamma}\left\{ \trace  a_{\theta \nu}(x)D^2u(x) + \cI_{\theta \nu}[x,u]\right\}.\]
 We denote $v=u/\delta.$ Following \cite{BMS22}, next we obtain the operator inequalities satisfied by $v.$
\begin{lemma}\label{L4.1}
Let $\Omega$ be bounded $C^2$ domain in $\Rd$. If $|\cL u| \leq K$ in $\Omega$ and $u=0$ in $\Omega^c,$ then we have
\begin{equation}\label{EL2.7A}
\begin{aligned}
&\cL v+2K_0 d^2 \frac{|D\delta|}{\delta}|Dv| \geq \frac{1}{\delta} \Big[-K-|v|(P^{+}+P^{+}_k)\delta-\sup_{\theta, \nu}Z_{\theta \nu}[v, \delta]\Big] ,\\
&\cL v-2K_0 d^2 \frac{|D\delta|}{\delta}|Dv| \leq \frac{1}{\delta} \Big[K-|v|(P^{-}+P^{-}_k)\delta-\inf_{\theta, \nu}Z_{\theta \nu}[v, \delta]\Big]
\end{aligned}
\end{equation}
for some $K_0,$ where \[Z_{\theta \nu}[v, \delta](x)=\int_{\Rd}(v(y)-v(x))(\delta(y)-\delta(x))N_{\theta \nu}(x,y-x)dy.\]
\end{lemma}
\begin{proof}
First note that, since $u\in C^1(\Omega)$ by \cref{TH2.1}, we have $v\in C^1(\Omega)$.
Therefore, $Z_{\theta \nu}[v, \delta]$ is continuous in $\Omega$.
Consider a test function $\psi\in C^2(\Omega)$ that touches 
$v$ from above at $x\in\Omega$. Define
\[
\psi_r(z)=\left\{
\begin{array}{lll}
\psi(z)\quad &\text{in}\; \sB_r(x),
\\
v(z)\quad &\text{in}\; \sB^c_r(x).
\end{array}
\right.
\]
By our assertion, we have $\psi_r\geq v$ for all $r$ small. To verify the first inequality in \eqref{EL2.7A} we must show that
\begin{align}\label{EL2.7B}
\cL [x, \psi_r] + 2 k_0 d^2\frac{|D\delta(x)|}{\delta(x)}\cdot |D \psi_r(x)|
\geq \frac{1}{\delta(x)}[-K - |v(x)| (\cP^+ + \cP^+_k)\delta(x) -\sup_{\theta, \nu}Z_{\theta \nu}[v,\delta](x)],
\end{align}
for some $r$ small. We define
\[
\tilde\psi_r(z)=\left\{
\begin{array}{lll}
\delta(z)\psi(z)\quad &\text{in}\; \sB_r(x),
\\
u(z)\quad &\text{in}\; \sB^c_r(x).
\end{array}
\right.
\]
Then, $\tilde\psi_r\geq u$ for all $r$ small. Since $|\cL u|\leq K$
and $\delta\psi_r=\tilde\psi_r$, we obtain at a point $x$ \\
\begin{align*}
-K\leq &\cL [x, \tilde\psi_r]\\
=&\sup_{\theta\in \Theta} \inf_{\nu \in \gamma}  \biggl[ \delta(x)\left(\trace a_{\theta\nu}(x)D^2\psi_r(x)+\cI_{\theta\nu} \psi_r(x)\right) + \psi_r(x)\left(\trace a_{\theta \nu}(x)D^2\delta(x)+\cI_{\theta \nu}\delta(x)\right) \\
&+\trace\big[\left(a_{\theta \nu}(x)+a^{T}_{\theta \nu}(x)\right)\cdot\left(D\delta(x)\otimes D\psi_r(x)\right)\big]+ Z_{\theta \nu}[\psi_r, \delta](x) \biggr]\\
&\leq \delta(x)\cL[x, \psi_r]+ \sup_{\theta, \nu} \biggl[ \psi_r(x) \left(\trace a_{\theta \nu}(x)D^2\delta(x)+\cI_{\theta \nu}\delta(x)\right) \\
&+\trace\left[\left(a_{\theta \nu}(x)+a^{T}_{\theta \nu}(x)\right)\cdot\left(D\delta(x)\otimes D\psi_r(x)\right)\right]+ Z_{\theta \nu}[\psi_r, \delta](x) \biggr]\\
&\leq \delta(x)\cL [x,\psi_r] + |v(x)| \left(\cP^{+}+\cP^{+}_k\right)\delta(x) +2K_0d^2 |D\delta(x)|\cdot |D\psi_r(x)| + \sup_{\theta, \nu}Z_{\theta \nu}[\psi_r, \delta](x),
\end{align*}
for all $r$ small and some constant $K_0$, where $D\delta(x)\otimes D\psi_r(x):=\left(\frac{\partial \delta}{\partial x_i}\cdot\frac{\partial \psi_r}{\partial x_j}\right)_{i,j}.$ Rearranging the terms we have
\begin{equation}\label{EL2.7C}
-K-|v(x)| \left(\cP^{+}+\cP^{+}_k\right)\delta(x)- \sup_{\theta, \nu}Z_{\theta \nu}[\psi_r, \delta](x)\leq \delta(x)\cL[x,\psi_r] + 2K_0 d^2|D\delta(x)|\cdot |D\psi_r(x)|.
\end{equation}
Let $r_1\leq r$. Since $\psi_r$ is decreasing with $r$, we get
from \eqref{EL2.7C} that
\begin{align*}
\delta(x)\cL[x, \psi_r] + 2K_0 d^2|D{\delta(x)}|\cdot |D\psi_r(x)|
&\geq \delta(x)\cL[x,\psi_{r_1}] + 2K_0 d^2|D{\delta(x)}|\cdot| D\psi_{r_1}(x)|
\\
&\geq \lim_{r_1\to 0}\left[-K-\left|v(x)\right| \left(\cP^{+}+\cP^{+}_k\right)\delta(x)- \sup_{\theta, \nu}Z_{\theta \nu}[\psi_{r_1}, \delta](x)\right]
\\
&= \left[-K-\left|v(x)\right| \left(\cP^{+}+\cP^{+}_k\right)\delta(x)- \sup_{\theta, \nu}Z_{\theta \nu}[v, \delta](x)\right],
\end{align*}
by dominated convergence theorem. This gives \eqref{EL2.7B}. Similarly we can verify the second inequality  of \eqref{EL2.7A}.
\end{proof}
Next we obtain a the following estimate on $v,$ away from the boundary. Denote $\Omega^\sigma= \{x\in \Omega\; :\; \dist(x, \Omega^c)\geq \sigma\}$. 
\begin{lemma}\label{L4.2}
Let $\Omega$ be bounded $C^2$ domain in $\Rd$. If $|\cL u| \leq K$ in $\Omega$ and $u=0$ in $\Omega^c,$ 
then for some constant $C$ it holds that
\begin{equation}\label{EL2.8A0}
\norm{D v}_{L^\infty(\Omega^\sigma)}\leq CK \sigma^{\upkappa-1}
\quad \text{for all}\; \sigma\in (0, 1),
\end{equation}
where $\upkappa$ is the H\"{o}lder exponent from \cref{T1.2}.
Furthermore, there exists $\eta\in (0,1)$ such that for any
$x\in \Omega^\sigma$ and $0<|x-y|\leq \sigma/8$ we have
\begin{align}\label{grad v estimate}
\frac{|D v(y)- D v(x)|}{|x-y|^\eta}\leq CK \sigma^{\upkappa-1-\eta}\quad \text{for all}\,\, \sigma\in (0,1).    
\end{align}
\end{lemma}
\begin{proof}
Using \cref{L4.1} we have
\begin{equation}\label{Eq4.5}
\begin{aligned}
&\cL v+2K_0 d^2 \frac{|D\delta|}{\delta}|Dv| \geq \frac{1}{\delta} \Big[-K-|v|(\cP^{+}+\cP^{+}_k)\delta-\sup_{\theta, \nu}Z_{\theta \nu}[v, \delta]\Big] ,\\
&\cL v-2K_0 d^2 \frac{|D\delta|}{\delta}|Dv| \leq \frac{1}{\delta} \Big[K-|v|(\cP^{-}+\cP^{-}_k)\delta-\inf_{\theta, \nu}Z_{\theta \nu}[v, \delta]\Big]
\end{aligned}
\end{equation}
in $\Omega$.
Fix a point $x_0\in \Omega^\sigma$ and define
$$w(x)=v(x)-v(x_0).$$
From \eqref{Eq4.5} we then obtain
\begin{equation}\label{Eq4.6}
\begin{aligned}
&\cL w+2K_0 d^2 \frac{|D\delta|}{\delta}|Dw| \geq  \Big[-\frac{1}{\delta} K-\ell_1 \Big] ,\\ 
&\cL w-2K_0 d^2 \frac{|D\delta|}{\delta}|Dw| \leq \Big[\frac{1}{\delta} K+ \ell_2 \Big]\\
\end{aligned}
\end{equation}
in $\Omega$, where
$$
\ell_1(x)= \frac{1}{\delta(x)} \left[ |w(x)|(\cP^{+}+\cP^{+}_k)\delta(x)+\sup_{\theta, \nu}Z_{\theta \nu}[w, \delta](x) + |v(x_0)|(\cP^{+}+\cP^{+}_k)\delta(x)\right],
$$
and
$$
\ell_2(x)= \frac{1}{\delta(x)} \left[ |w(x)|(\cP^{-}+\cP^{-}_k)\delta(x)-\inf_{\theta, \nu}Z_{\theta \nu}[w, \delta] (x) -|v(x_0)|(\cP^{-}+\cP^{-}_k)\delta(x)  \right].
$$
We set $r=\frac{\sigma}{2}$ and claim that
\begin{equation}\label{EL2.8C}
\norm{\ell_{i}}_{L^\infty(B_r(x_0))}\leq \kappa_1 \sigma^{\upkappa-2},
\quad \text{for all}\; \sigma\in (0, 1) \; \text{and} \; i=1,2 ,
\end{equation}
for some constant $\kappa_1$.
Let us denote by
$$ \xi_1^{\pm}=\frac{|w(x)| (\cP^{\pm}+\cP^{\pm}_k)\delta}{\delta},
\quad \xi_2=\frac{1}{\delta} \sup_{\theta, \nu}Z_{\theta \nu}[w, \delta],
\quad  \xi_3^{\pm} = \frac{|v(x_0)|(\cP^{\pm}+\cP^{\pm}_k)\delta}{\delta}, \quad \xi_4=\frac{1}{\delta} \inf_{\theta, \nu}Z_{\theta \nu}[v, \delta] .$$
Recall that $\upkappa\in (0,  \hat{\alpha}).$ 
Since 
$$\norm{\cP^{\pm} \delta}_{L^\infty(\Omega)}<\infty\quad
\text{and}\quad \norm{ \cP^{\pm}_k \delta}_{L^\infty(\Omega_\sigma)}
\lesssim \left(1+\Ind_{(1,2)}(\alpha)\delta^{1-\alpha}\right)
$$
(cf \cref{Lem3.4} ), and
$$\norm{v}_{L^\infty}(\Rd)<\infty, 
\quad \norm{w}_{L^\infty(B_r(x_0))}\lesssim r^{\upkappa},$$
it follows that
$$\norm{\xi_3^{\pm}}_{L^\infty(B_r(x_0))}\lesssim 
\left.\begin{cases}
\frac{1}{\sigma} & \quad \text{if}\; \alpha\in (0, 1],
\\
\frac{1}{\sigma^{\alpha}}& \quad \text{if}\; \alpha\in(1,2)
\end{cases}
\right\}
\lesssim 
\sigma^{ \upkappa-2},$$
and
$$\norm{\xi_1^{\pm}}_{L^\infty(B_r(x_0))}\lesssim
\left.\begin{cases}
\frac{\sigma^{\upkappa}}{\delta} & \quad \text{if}\; \alpha\in (0, 1],
\\
\frac{\sigma^{\upkappa}}{\delta^{\alpha}}& \quad \text{if}\; \alpha\in(1,2)
\end{cases}
\right\}
\lesssim \sigma^{\upkappa-2}.$$
Next we estimate $\xi_2$ and $\xi_4$. Let $x\in \sB_r(x_0)$.
Denote by $\hat{r}=\delta(x)/4$. Note that 
$$\delta(x)\geq \delta(x_0)-|x-x_0|\geq 2r-r =r
\Rightarrow \hat{r}\geq r/4. $$
Since $u\in C^1(\Omega)$ by \cref{TH2.1} and $|u| \leq C \delta$ in $\Rd$ by \cref{Barrier1}. Thus we have
\begin{equation}\label{E5.8}
    |D v|\leq \left|\frac{D u}{\delta}\right|+ \left|\frac{uD \delta}{\delta^2}\right|\lesssim \frac{1}{\delta(x)}\quad \text{in}\; \sB_{\hat{r}}(x).
    \end{equation}
Now we calculate
\begin{align*}
|Z_{\theta \nu}[w, \delta](x)|\leq \int_{\Rd}|\delta(x)-\delta(y)||v(x)-v(y)|k(y-x)\D{y}
&=\int_{\sB_{\hat{r}}(x)}+\int_{\sB_1(x)\setminus \sB_{\hat r}(x)}+ \int_{\sB^c_1(x)}\\
&=I_1+I_2+I_3.
\end{align*}
To estimate $I_1,$ first we consider $\alpha\leq 1.$ Since $\delta$ is Lipschitz continuous and $v$ bounded on $\Rd,$ $I_1$ can be written as
\begin{align*}
    I_1&=\int_{\sB_{\hat{r}}(x)} \frac{|\delta(x)-\delta(y)|}{|x-y|} |v(x)-v(y)|\cdot |x-y|k(y-x)\D{y}\\
    &\lesssim \int_{\sB_{\hat{r}}(x)} |x-y|^{\alpha}k(y-x)\D{y} \leq \int_{\Rd}(1\wedge |z|^{\alpha})k(z)\D{z} . 
\end{align*}
For $\alpha \in (1,2),$ using the Lipschitz continuity of $\delta$ and \eqref{E5.8} we get
\begin{align*}
    I_1&=\int_{\sB_{\hat{r}}(x)} \frac{|\delta(x)-\delta(y)|}{|x-y|} \cdot \frac{|v(x)-v(y)|}{|x-y|}\cdot |x-y|^{\alpha} |x-y|^{2-\alpha}k(y-x)\D{y} \\
    &\lesssim  \frac{\hat{r}^{2-\alpha}}{\delta(x)} \int_{\sB_{\hat{r}}(x)}  |x-y|^{\alpha}k(y-x)\D{y} \lesssim \delta(x)^{1-\alpha}\int_{\Rd}(1\wedge |z|^{\alpha}) k(z)\D{z} \lesssim \sigma^{\upkappa-1}.
\end{align*}
Bounds on $I_2$ can be computed as follows: for
$\alpha\leq 1,$ we write
\begin{align*}
    I_2=\int_{\sB_1(x)\setminus \sB_{\hat r}(x)} |\delta(x)-\delta(y)||v(x)-v(y)|k(y-x)\D{y}
&\lesssim\int_{\sB_1(x)\setminus \sB_{\hat r}(x)} |x-y|^{\alpha} k(y-x) \D{y}\\
&\lesssim \int_{\Rd}(1\wedge |z|^\alpha)k(z)\D{z}. 
\end{align*} 
In the second line of the above inequality, we used
\[|\delta(x)-\delta(y)|\lesssim|x-y|\,\,\, \text{and}\,\,\,\,||v||_{L^{\infty}(\Rd)} <\infty.\]
For $\alpha\in (1,2)$ we can compute $I_2$ as
\begin{align*}
    \int_{\sB_1(x)\setminus \sB_{\hat r}(x)}|\delta(x)-\delta(y)||v(x)-v(y)|k(y-x) \D{y} &\lesssim \int_{\sB_1(x)\setminus \sB_{\hat r}(x)} |x-y|^{1-\alpha}\cdot|x-y|^{\alpha} k(y-x)\D{y}\\
&\lesssim \delta(x)^{1-\alpha} \int_{\Rd}(1\wedge |z|^\alpha)k(z)\D{z} \lesssim \sigma^{\upkappa-1}.
\end{align*} Moreover, since $\delta$ and $v$ are bounded in $\Rd$, we get $I_3\leq \kappa_3.$
Combining the above estimates we obtain
$$\norm{\xi_i}_{L^\infty \sB_r(x_0)}\lesssim \sigma^{\upkappa-2}\,\,\, \text{for}\,\,\, i=2,4.$$
Thus the claim \eqref{EL2.8C} is established.

Let us now define $\zeta(z)=w(\frac{r}{2}z + x_0)$. Letting
$b(z)=\frac{D\delta(\frac{r}{2}z+x_0)}{2 \delta(\frac{r}{2}z + x_0)}$ it follows from \eqref{Eq4.6} that
\begin{align}\label{EL2.8D}
&\tilde{\cL}^r \zeta +K_0d^2r b(z) \cdot |D\zeta| \geq -\frac{r^2}{4}\left[\frac{1}{\delta}K+l_1\right]\left(\frac{r}{2}z+x_0\right)\\
&\tilde{\cL}^r \zeta -K_0d^2r b(z) \cdot |D\zeta| \leq \frac{r^2}{4}\left[\frac{1}{\delta}K+l_2\right]\left(\frac{r}{2}z+x_0\right)\nonumber
\end{align}
in $\sB_2(0),$ where
\[\tilde{\cL}^r [x,u]:=\sup_{\theta \in \Theta}\inf_{\nu \in \Gamma}\left\{\trace\left(a_{\theta \nu}\left(\frac{r}{2}x+x_0\right)D^2u(x)\right)+\tilde{\cI}^r_{\theta\nu}[x,u]\right\}\] and $\tilde{\cI}^r_{\theta\nu}$ is given by
\[\tilde{\cI}^r_{\theta \nu}[x,f]=\int_{\Rd}\left(f(x+y)-f(x)-\Ind_{B_{\frac{1}{r}}(y)} \nabla f(x)\cdot y\right)\left(\frac{r}{2}\right)^{d+2}N_{\theta \nu}\left(\frac{r}{2}x+x_0,ry\right)\D{y}.\]

Consider a cut-off function $\varphi$ satisfying
$\varphi=1$ in $\sB_{3/2}$ and $\varphi=0$ in $\sB^c_2$. 
Defining $\tilde\zeta=\zeta\varphi$ we get from \eqref{EL2.8D} that
\begin{align*}
    &\tilde{\cL}^r [z, \tilde{\zeta}]+K_0d^2rb(z).|D\tilde{\zeta}(z)| \geq -\frac{r^2}{4}\left[\frac{K}{\delta}+|l_1|\right]\left(\frac{r}{2}z+x_0\right)-\left|\sup_{\theta \in \Theta}\inf_{\nu \in \Gamma}\tilde{\cI}^r_{\theta \nu}[z, (\varphi-1)\zeta]\right|\\
    &\tilde{\cL}^r [z, \tilde{\zeta}]-K_0d^2rb(z).|D\tilde{\zeta}(z)| \leq \frac{r^2}{4}\left[\frac{K}{\delta}+|l_1|\right]\left(\frac{r}{2}z+x_0\right)-\left|\sup_{\theta \in \Theta}\inf_{\nu \in \Gamma}\tilde{\cI}^r_{\theta \nu}[z, (\varphi-1)\zeta]\right|
\end{align*}
in $\sB_1$. Since 
$$\norm{rb}_{L^\infty(\sB_1(0))}\leq \kappa_3\quad \text{for all}\;
\sigma\in(0,1),$$
applying \cref{TH2.1} we obtain,  for some $\eta\in(0,1)$,
\begin{equation}\label{EL2.8E}
\norm{D \zeta}_{C^\eta(\sB_{1/2}(0))}
\leq \kappa_6 \left(\norm{\tilde\zeta}_{L^\infty(\Rd)} +\kappa_4 \sigma + \kappa_5 \sigma^{\upkappa}\right),
\end{equation}
for some constant $\kappa_6$ independent of $\sigma\in (0,1),$ where we used 
\begin{align*}
    &\left|\tilde{\cI}^r_{\theta \nu}[z, (\varphi-1)\zeta]\right| \lesssim \sigma \,\,\,\,\, (\text{cf. the proof of \cref{T1.1}})\,\,\,\, \text{and}\,\,\,\, |l_1|(\frac{r}{2}\cdot +x_0) \lesssim \sigma^{\upkappa-2}.
    \end{align*}
Since $v$ is in $C^\upkappa(\Rd)$, it follows that
$$\norm{\tilde\zeta}_{L^\infty(\Rd)}=\norm{\tilde\zeta}_{L^\infty(\sB_2)}\leq \norm{\zeta}_{L^\infty(\sB_2)}\lesssim r^\upkappa.$$

Putting these estimates in \eqref{EL2.8E} and calculating the gradient at $z=0$ we obtain
$$|D v(x_0)|\lesssim \sigma^{\upkappa-1},$$
for all $\sigma\in (0,1)$. This proves the H\"{o}lder estimate \eqref{EL2.8A0}.

For the second part, compute the H\"{o}lder ratio with
$D\zeta(0)-D\zeta(z)$ where
$z=\frac{2}{r}(y-x_0)$ for $|x_0-y|\leq \sigma/8$.
This completes the proof.

\end{proof}
Now we can complete the proof of \cref{T1.3}. If $u$ is a solution of the operator inequalities \eqref{Eq-1}, then using \cref{T1.1} we have $|\cL u| \leq CK $. Finally, the proof can be obtained following the same lines as in \cite[Theorem 1.3]{BMS22}. We present it here for the sake of completeness.
\begin{proof}[Proof of \cref{T1.3}]
Since $u=v\delta$ it follows that
$$D u = vD \delta + \delta D v.$$
Since $\delta\in C^{2}(\overline\Omega)$, it follows from
\cref{T1.2} that
$v D\delta\in C^\upkappa(\overline\Omega)$. Thus, we only need to concentrate on $\vartheta=\delta D v$. Consider $\eta$ from
\cref{L4.2} and, without loss of generality, we fix $\eta\in (0, \upkappa)$.

For $|x-y|\geq \frac{1}{8}(\delta(x)\vee\delta(y))$ it follows from
\cref{EL2.8A0} that
$$\frac{|\vartheta(x)-\vartheta(y)|}{|x-y|^\eta}
\leq CK (\delta^\upkappa(x) + \delta^{\upkappa}(y))(\delta(x)\vee\delta(y))^{-\eta}\leq 2CK.
$$
So consider the case $|x-y|< \frac{1}{8}(\delta(x)\vee\delta(y))$.
Without loss of generality, we may assume that
$|x-y|< \frac{1}{8}\delta(x)$. Then we have
\begin{align*}
&\frac{9}{8}\delta(x)\geq |x-y|+\delta(x)\geq \delta(y),\\
&\delta(y)\geq \delta(x)-|x-y|\geq \frac{7}{8}\delta(x).
\end{align*}
By \eqref{grad v estimate} of \cref{L4.2} and the above inequalities,  it follows
\begin{align*}
\frac{|\vartheta(x)-\vartheta(y)|}{|x-y|^\eta}
&\leq |D v(x)|\frac{|\delta(x)-\delta(y)|}{|x-y|^\eta}
+\delta(y)\frac{|D v(x)-D v(y)|}{|x-y|^\eta}
\\
&\lesssim \delta(y)^{\upkappa-1}(\delta(x))^{1-\eta} +
\delta(y) [\delta(x)]^{\upkappa-1-\eta}\\
&\lesssim \delta(x)^{\upkappa-\eta} +
\delta(y) [\delta(x)]^{\upkappa-1-\eta}=\delta(x)^{\upkappa-\eta}\left(1+\frac{\delta(y)}{\delta(x)}\right)\\
&\leq CK(\diam \Omega)^{\upkappa-\eta}.
\end{align*}
This completes the proof.
\end{proof}

\begin{appendices}
  \crefalias{section}{appsec}
\section{Appendix}\label{appendix}
In this section, we aim to present a proof of \cref{TH2.1}. For this purpose,
we first introduce the scaled operator. Let $x_0\in \Omega$ and $r>0,$ we define the doubly scaled operator as
\begin{align}\label{eqA1}
    \cL^{r,s}(x_0)[x,u]=\sup_{\theta \in \Theta} \inf_{\nu \in \Gamma}\left\{\trace a_{\theta \nu}(sr(x-x_0)+sx_0)D^2u(x) +\cI^{r,s}_{\theta \nu}(x_0)[x,u]\right\}
\end{align}
where
\begin{align*}
    \cI^{r,s}_{\theta \nu}(x_0)[x,u]=\int_{\Rd}(u(x+y) -u(x) - \Ind_{\sB_{\frac{1}{sr}}}(y) \grad u(x) \cdot y  )r^{d+2}(s^{d+2}N_{\theta \nu}(rs(x-x_0)+sx_0, sry) \D{y}.
\end{align*}
Further, we define
\begin{equation}\label{eqA2}
    \cL^{0, s}(x_0)[x,u]:= \displaystyle{\sup_{\theta \in \Theta}} \inf_{\nu \in \Gamma} \left\{\trace a_{\theta \nu}(sx_0) D^2 u(x)\right\}.
\end{equation}
Now we give the definition of weak convergence of operators.
\begin{definition}\label{conv def}
    Let $\Omega \subset \Rd$ be open and $0<r<1.$ A sequence of operators $\cL^m$ is said to converge weakly to $\cL$ in $\Omega$, if for any test function $\varphi \in L^{\infty}(\Rd)\cap C^2(\sB_r(x_0))$ for some $\sB_r(x_0) \subset \Omega,$ we have
    \begin{equation*}
        \cL^m[x,\varphi] \to \cL[x, \varphi]\,\,\,\,\,\, \text{uniformly}\,\,\, \text{in}\,\,\, \sB_{\frac{r}{2}}(x_0) \,\,\, \text{as}\,\,\, m \to \infty.
    \end{equation*}
\end{definition}
The next lemma is a slightly modified version of \cite[Lemma 4.1]{MZ21} which can be proved by similar arguments.
\begin{lemma}\label{A1}
    For any $x_0\in \sB_1,$ $r>0$ and $0<s<1,$ Let $\cL^{r,s}(x_0)$ and $\cL^{0,s}(x_0)$ is given by \eqref{eqA1} and \eqref{eqA2} respectively where the \cref{Assmp-1} are satisfied by the corresponding coefficients with $\Omega=\sB_2.$ Moreover, for  given $M, \varepsilon >0$ and a modulus of continuity $\rho,$ there exists $r_0, \eta >0$ independent of $x_0$ and $s$ such that if
   \begin{enumerate}
        \item [(i)] $r<r_0$,  $\cL^{0,s}(x_0)[x,u]=0$ in $\sB_1,$
        \item [(ii)] \begin{align*}
    \cL^{r,s}(x_0)[x,u]+C_0rs|Du(x)| \geq -\eta\,\,\,\, &in\,\,\, \sB_1\\
    \cL^{r,s}(x_0)[x,u]-C_0rs|Du(x)| \leq \eta\,\,\,\, &in\,\,\, \sB_1,\\
    u=v\,\,\, &in \,\,\, \partial \sB_1.
\end{align*}
\item[(iii)] $|u(x)|+|v(x)| \leq M$ in $\Rd$ and $|u(x)-u(y)|+|v(x)-v(y)| \leq \rho(|x-y|)$ for all $x,y \in \overline{\sB}_1,$
    \end{enumerate}
    then we have
    \[|u-v|\leq \varepsilon\] in $\sB_1.$
\end{lemma}

It is worth mentioning that in \cite{MZ21} the authors have set a uniform continuity assumption on the nonlocal kernels $N_{\theta\nu} (x,y)$ ( for the precise assumption, see Assumption (C) of \cite[p. 391]{MZ21} ) which is a standard assumption to make for the stability property of viscosity solutions. Namely, if we have a sequence of integro-differential operators $\cL^m$ converging weakly to $\cL$ in $\Omega$ and a sequence of subsolutions (or supersolutions) in $\Omega$ converging 
locally uniformly on any compact subset of $\Omega$, then the limit is also a subsolution (or supersolution) with respect to  $\cL$. However, in the case of the operator $\cL^{r,s}$ defined in \eqref{eqA1}, the nonlocal term $ \cI^{r,s}_{\theta \nu}$ can be treated as a lower order term that converges to zero as $r \to 0$ without any kind of continuity assumptions on nonlocal kernels $N_{\theta \nu}$.

Now we give the proof of \cref{TH2.1}.
\begin{proof}[Proof of \cref{TH2.1}]
We will closely follow the proof of \cite[Theorem 4.1]{MZ21}. Fix any $x_0 \in \sB_1$, let $\cL^{r_k,s}(x_0)$ and $\cL^{0,s}(x_0)$ is given by \eqref{eqA1} and \eqref{eqA2} respectively. Then by \cite[Lemma 3.1]{MZ21} as $r_k \to 0,$ we have
\[\cL^{r_k, s}(x_0) \to \cL^{0,s}(x_0),\]
 in the sense of \cref{conv def}.
By interior regularity \cite[Corollary 5.7]{CCbook}, $\cL^{0,s}(x_0)$ has $C^{1,\beta}$ estimate for an universal constant $\beta >0.$ Now without loss of any generality we may assume that $x_0=0$. Also dividing $u$ by $||u||_{L^{\infty}(\Rd)} + K$ in \eqref{eqq2.1} we may assume that $K=1$ and $||u||_{L^{\infty}(\Rd)} \leq 1.$ \\
Using the H\"older regularity  \cite[Lemma~2.1]{MZ21}, we have $u \in C^{\beta}(\sB_1).$ Following \cite[Theorem 52]{CS11}, we will show that there exists $\delta, \mu \in (0, \frac{1}{4}),$ independent of $s$ and a sequence of linear functions $l_k(x)=a_k+b_k x$  such that
\begin{equation}\label{induction}
\begin{aligned}
    \begin{cases}
    (i)\,\,\, \displaystyle{\sup_{B_{2\delta \nu ^k}}}|u-l_k| \leq \mu^{k(1+\gamma)} \;, \\
    (ii)\,\,\, \left|a_k-a_{k-1}\right| \leq \mu^{(k-1)(1+\gamma)} \;,\\
    (iii)\,\,\, \mu^{k-1} |b_k-b_{k-1}|\leq C \mu^{(k-1)(1+\gamma)} \;,\\
    (iv)\,\,\, \left|u-l_k\right| \leq \mu ^{-k(\gamma^{\prime}-\gamma)} \delta^{-(1+\gamma^{\prime})} |x|^{1+\gamma^{\prime}} \,\,\,\, \text{for}\,\,\, x\in \sB^c_{2\delta \mu^k } \;,
    \end{cases}
    \end{aligned}
\end{equation}
where $0<\gamma<\gamma^{\prime}<\beta $ do not depend on $s$. We plan to proceed by induction, when $k=0,$ since $||u||_{L^{\infty}(\Rd)}\leq 1$, \eqref{induction} holds with $l_{-1} = l_0=0$. Assume \eqref{induction}  holds for some $k$ and we shall show \eqref{induction} for $k+1.$

Let $\xi : \mathbb{R}^d \to [0,1]$ be a continuous function such that
\begin{align*}
    \xi(x)=
    \begin{cases}
    1 \,\,\, \text{for} \,\,\, x \in \sB_3 ,\\
    0 \,\,\, \text{for} \,\,\, x \in \sB^c_4 .
 \end{cases} 
\end{align*}
Let us define
\begin{align*}
    w_k(x)=\frac{(u-\xi l_k)(\delta \mu^{k}x)}{\mu ^{k(1+\gamma)}}.
\end{align*}
We claim that there exists a universal constant $C>0$, such that for all $k,$ we have
\begin{equation}\label{vis}
\begin{aligned}
    &\cL^{r_k,s} [x, w_k]-C_0 r_k s |Dw_k(x)| \leq C\delta^2 \mu^{k(1-\gamma)} \leq C\delta^2 ,\\
    &\cL^{r_k,s} [x, w_k]+C_0 r_k s |Dw_k(x)| \geq -C\delta^2 \mu^{k(1-\gamma)} \geq -C\delta^2 ,
\end{aligned}
\end{equation}
in $\sB_2$ in viscosity sense. Let $\phi \in C^2(\sB_2)\cap C(\Rd)$ which touches $w_k$ from below at $x^{\prime}$ in $\sB_2.$ Let 
\[\psi(x):=\mu^{k(1+\gamma)} \phi \left(\frac{x}{\delta \mu ^k}\right)+ \xi l_k(x).\]
Then $\psi \in C^2(\sB_{2\delta \mu^k}) \cap C(\Rd)$ is bounded and touches $u$ from below at $\delta \mu^k x^{\prime}.$ 
Taking $r_k=\delta \mu^k,$ we have 
\[\cI^{r_k, s}_{\theta \nu}[x^{\prime}, \phi]=\delta^2 \mu^{k(1-\gamma)} \cI^s_{\theta \nu}[r_k x^{\prime}, \psi-\xi l_k].\]
Thus we get
\begin{equation*}
\begin{aligned}
    &\cL^{r_k, s}[x^{\prime}, \phi]-C_0 r_k s|D\phi(x^{\prime})|\\
    &=\delta^2 \mu^{k(1-\gamma)}\Big[\sup_{\theta \in \Theta}\inf_{ \nu \in \Gamma} \left\{\trace a_{\theta \nu}(sr_k x^{\prime})D^2 \psi(r_kx^{\prime})+\cI^s_{\theta \nu}[r_k x^{\prime}, \psi-\xi l_k]\right\}-sC_0|D\psi(r_k x^{\prime})-b_k|\Big]\\
    &\leq\delta^2 \mu^{k(1-\gamma)} \Big[\cL^s[r_k x^{\prime}, \psi]-sC_0|D\psi(r_k x^{\prime})|+\sup_{\theta \in \Theta} \inf_{\nu \in \Gamma} \{-\cI^s_{\theta \nu}[r_k x^{\prime}, \xi l_k]\}+s C_0|b_k|\Big]\\
    &\leq C \delta^2 \mu^{k(1-\gamma)} \leq C\delta^2.
\end{aligned}
\end{equation*}
 In the second last inequality we use that $$ \cL^s[x,u]-C_0 s |Du(x)| \leq 1,$$ 
 and $|a_k|$, $|b_k|$ are uniformly bounded and for all $x^{\prime} \in \sB_2$,  $\displaystyle{\sup_{\theta \in \Theta} \inf_{\nu \in \Gamma}} \{-\cI^s_{\theta \nu}[r_k x^{\prime}, \xi l_k]\}$ is bounded independent of $s$ and $k$ .
Thus we have proved 
\[\cL^{r_k,s} [x, w_k]-C_0 r_k s |Dw_k(x)| \leq C\delta^2\,\,\, \text{in}\,\,\, \sB_2 ,\] in viscosity sense. Similarly the other inequality in \eqref{vis} can be proven.

Define $w^{\prime}_k(x):=\max \left\{\min\left\{w_k(x),1\right\}, -1 \right\}.$ We see that $w^{\prime}_k$ is uniformly bounded independent of $k.$ We claim that in $\sB_{\frac{3}{2}}$
\begin{equation}\label{eqA5}
    \begin{aligned}
        &\cL^{r_k,s} [x, w^{\prime}_k]-C_0 r_k s |Dw^{\prime}_k(x)|\leq C\delta^2 + \omega_1(\delta) ,\\
    &\cL^{r_k,s} [x, w^{\prime}_k]+C_0 r_k s |Dw^{\prime}_k(x)| \geq -C\delta^2-\omega_1(\delta)
    \end{aligned}
\end{equation}
Now take any bounded $\phi \in C^2(\sB_2)\cap C(\Rd)$ that touches $w^{\prime}_k$ from below at $x^{\prime}$ in $\sB_{3/2}.$ By the definition of $w^{\prime}_k,$ in $\sB_2$ we have $|w_k|=|w^{\prime}_k| \leq 1$ and $\phi$ touches $w_k$ from below at $x^{\prime}.$ Thus by \cref{viscosity defn} of viscosity supersolution we have 
\begin{align*}
&\sup_{\theta \in \Theta} \inf_{\nu \in \Gamma}\Big\{\trace a_{\theta \nu} (sr_k x^{\prime})D^2\phi(x^{\prime})\\
&+\int_{\sB_{1/2}}(\phi(x^{\prime}+z)+\phi(x^{\prime})-\Ind_{\sB_{\frac{1}{r_k s}}}(z)D\phi(x^{\prime})\cdot z)(r_ks)^{d+2}N_{\theta \nu}(r_k sx, sr_kz)\D{z}\\
&-\int_{\Rd \setminus \sB_{1/2}}(w_k(x^{\prime}+z)-w^{\prime}_k(x^{\prime}+z)\\
&+w^{\prime}_k(x^{\prime}+z)-\phi(x^{\prime})-\Ind_{\sB_{\frac{1}{r_k s}}}(z)D\phi(x^{\prime})\cdot z))(r_ks)^{d+2}N_{\theta \nu}(r_k s x, sr_kz)\D{z}\Big\}\\
    &-C_0r_ks |D\phi(x^{\prime})| \leq C\delta^2
\end{align*}
Now we use the bounds on the kernel to get the following estimate:
\begin{equation*}
    \begin{aligned}
        \cL^{r_k, s}[x, w^{\prime}_k]-C_0r_ks|Dw^{\prime}_k(x)| &\leq \int_{\Rd\setminus \sB_{1/2}}\left|w_k(x^{\prime}+z)-w^{\prime}_k(x^{\prime}+z)\right| (r_ks)^{d+2} k(r_k sz)\D{z}+ C \delta^2.
    \end{aligned}
\end{equation*}
in the viscosity sense. By the inductive assumptions, we have $a_k $ and $b_k$ uniformly bounded. Since $\left|\left|u\right|\right|_{L^{\infty}(\Rd)} \leq 1$ and $\xi l_k$ is uniformly bounded, $\left|w_k\right| \leq C \mu^{-k (1+\gamma)} $ in $\Rd.$ Using $(iv)$ from \eqref{induction} we have
\[|w_k(x)|=\frac{(u-\xi l_k)(r_kx)}{\mu^{k(1+\gamma)}} \leq \left(\frac{1}{r_k}\right)^{1+\gamma^{\prime}}|r_k x|^{1+\gamma^{\prime}}=|x|^{1+\gamma^{\prime}} ,\] for any $x\in \sB^c_2 \cap \sB_{\frac{2}{r_k}}.$ Again for any $x \in \sB^c_{2/r_k}, $ we find \[|w_k(x)|\leq C\mu^{-k(1+\gamma^{\prime})}\cdot \mu^{-k(\gamma-\gamma^{\prime})} \leq C\mu^{-k(1+\gamma^{\prime})} \leq C\frac{\delta^{1+\gamma^{\prime}}}{2}|x|^{1+\gamma^{\prime}}\leq C|x|^{1+\gamma^{\prime}} .\]
Now, since $w^{\prime}_k$ is uniformly bounded, we have for $x\in \sB^c_2,$ 
\begin{equation}\label{eqA6}
  |w_k|+|w^{\prime}_k-w_k|\leq C\min\{|x|^{1+\gamma^{\prime}}, \mu ^{-k(1+\gamma)}\}.  
\end{equation}
For $x^{\prime}\in \sB_{3/2},$ using \eqref{eqA6} we have the following estimate.

\begin{align*}
&\int_{\Rd}\left|w_k(x^{\prime}+z)-w^{\prime}_k(x^{\prime}+z)\right| (r_k s)^{d+2} k(r_k sz) \D{z} \\
&\leq \int_{\{z:|x^{\prime}+z|\geq 2\}\cap \sB_{1/r_k}}\left|w_k-w^{\prime}_k\right|(x^{\prime}+z)(r_k s)^{d+2} k(r_k sz) \D{z} + \delta^2 \mu^{k(1-\gamma)}\int_{\sB^c_{\frac{1}{r_k}}} \frac{(r_k s)^{d+2} k(r_k sz) }{(\delta \mu^{-k})^2} \D{z}\\
&\leq C \Big[ \int_{\sB^c_{1/2}\cap \sB_{\frac{1}{\sqrt{r_k}}}}|z|^2 (r_k s)^{d+2} k(r_k sz) \D{z} + r^{\frac{(1-\gamma^{\prime})}{2}}_k \int_{\sB^c_{\frac{1}{\sqrt{r_k}}}\cap \sB_{\frac{1}{r_k}}} |z|^2 (r_k s)^{d+2} k(r_k sz) \D{z}\\
&+ \delta^2 \mu^{k(1-\gamma)}\int_{\sB^c_{s}} s^2k(z) \D{z} \Big] \\ 
      &\leq C \Big[\int_{\sB_{\sqrt{r_k}}}|y|^2 k(y) \D{y} + (r^{\frac{(1-\gamma^{\prime})}{2}}_k + \delta^2 \mu^{k(1-\gamma)} )  \int_{\Rd}(1\wedge|y|^2) k(y)\D{y}\Big] . \\
        \end{align*}
Hence,
\begin{equation*}
    \begin{aligned}
    \int_{\Rd}\left|w_k(x^{\prime}+z)-w^{\prime}_k(x^{\prime}+z)\right| k^{r_k, s}(z) \D{z} \leq \tilde{C}\left(\int_{\sB_{\sqrt{\delta}}}|y|^2k(y) \D{y}+\delta^{\frac{1-\gamma^{\prime}}{2}}+\delta^2\right)=\omega_1(\delta)
    \end{aligned}
\end{equation*}
where $\omega_1(\delta) \to 0$ as $\delta \to 0.$ Therefore we proved $\cL^{r_k,s} [x, w^{\prime}_k]-C_0 r_k s |Dw^{\prime}_k(x)|\leq C\delta^2 + \omega_1(\delta).$ The other inequality of \eqref{eqA5} can be proved in a similar manner. 

Since $w^{\prime}_k$ satisfies the equation \eqref{eqA5}, by \cite[Lemma 2.1]{MZ21} we have $||w^{\prime}_k||_{C^{\beta}(\overline{\sB_1})} \leq M_1$ for some $M_1$ independent of $k, s.$ Now we consider the a function $h$ which solves
\begin{equation*}
    \begin{aligned}
        &\cL^{0,s}(x_0)[x, h]=0\,\,\,\, &in \,\,\, \sB_1\\
        &h=w^{\prime}_k\,\,\,\, &on \,\,\, \partial \sB_1.
    \end{aligned}
\end{equation*}
Existence of such $h$ can be seen from \cite[Theorem 1]{S2010}. Moreover, using \cite[Theorem 2]{S2010} we have $||h||_{C^{\alpha}(\overline{\sB_1})} \leq M_2$ where $\alpha < \frac{\beta}{2}$ and $M_2$ is independent of $k,s.$ Now for any $0<\varepsilon <1,$ let $r_0:=r_0(\varepsilon)$ and $\eta:=\eta(\varepsilon)$ as given in \cref{A1}. Also for $x\in \sB_1$ and $\delta:=\delta(\varepsilon) \leq r_0,$ we have 
\begin{equation*}
    \begin{aligned}
        &\cL^{r_k,s} [x, w^{\prime}_k]+C_0 r_k s |Dw^{\prime}_k(x)|\geq -\eta ,\\
    &\cL^{r_k,s} [x, w^{\prime}_k]-C_0 r_k s |Dw^{\prime}_k(x)| \leq \eta.
    \end{aligned}
\end{equation*}
Therefore by \cref{A1}, we conclude $|w^{\prime}_k-h| \leq \varepsilon$ in $\sB_1.$ Again by using \cite[Corollary 5.7]{CCbook}, we have $h \in C^{1, \beta}(\sB_{1/2})$ and we can take a linear part $l(x):=a+bx$ of $h$ at the origin. By $C^{1, \beta}$ estimate of $\cL^{0, s}(x_0)$ and $|w^{\prime}_k| \leq 1$ in $\sB_1$ we obtain that the coefficients of $l,$ i.e, $a, b$ are bounded independent of $k,s.$ Further for $x\in \sB_{1/2},$ we have \[|h(x)-l(x)| \leq C_1|x|^{1+\beta} ,\] where $C_1$ is independent of $k,s.$ Hence using the previous estimate we get
\[ |w^{\prime}_k(x)-l(x)| \leq \epsilon + C_1 |x|^{1+\beta} \,\,\,\, in\,\, \sB_{1/2}.\] Again using \eqref{eqA6} and $|w_k|\leq 1$ in $\sB_2$ we have
\begin{equation*}
    \begin{aligned}
        &|w_k(x)-l(x)| \leq 1+|a|+|b| \leq C_2\,\,\, in \,\,\, \sB_1,\\
        &|w_k(x)-\xi(\delta \mu^k x)l(x)| \leq C|x|^{1+\gamma^{\prime}}+C_3 |x|\,\,\, in \,\,\, \sB^c_1.
    \end{aligned}
\end{equation*}
Next defining 
\begin{equation*}
\begin{aligned}
    &l_{k+1}(x):=l_k(x)+\mu^{k(1+\gamma)}l\left(\delta^{-1}\mu^{-k}x\right),\\
    &w_{k+1}(x):=\frac{(u-\xi l_{k+1})(\delta\mu^{k+1}x)}{\mu^{(k+1)(1+\gamma)}},
    \end{aligned}
\end{equation*}
and following the proof of \cite[Theorem 4.1]{MZ21} we conclude that \eqref{induction} holds for $k+1.$ This completes the proof.
\end{proof}
\end{appendices}

\noindent\textbf{Acknowledgement.}
We thank Anup Biswas for several helpful discussions during the preparation of this article.  Mitesh Modasiya is partially supported by CSIR PhD fellowship (File no. 09/936(0200)/2018-EMR-I). We also thank the referee for his/ her useful comments, which improved the manuscript.

\bigskip
\noindent\textbf{Conflict of interest.} On behalf of all authors, the corresponding author states that there is no conflict of interest.

\bigskip
\noindent\textbf{Data availability.} This manuscript has no associated data.

%


\begin{thebibliography}{10}
\bibitem{AV19}N. Abatangelo and E. Valdinoci: Getting Acquainted with the Fractional Laplacian. \emph{In: Dipierro, S. (eds) Contemporary Research in Elliptic PDEs and Related Topics. Springer INdAM Series,} vol 33. (2019) Springer, Cham.

\bibitem{A}D. Applebaum: Lévy processes and stochastic calculus. Second edition. Cambridge Studies in Advanced Mathematics, 116. Cambridge University Press, Cambridge, 2009. xxx+460 pp. ISBN: 978-0-521-73865-1 


\bibitem{Barles12}
G. Barles, E. Chasseigne and C. Imbert: 
Lipschitz regularity of solutions for mixed integro-differential equations. \emph{J. Differential Equations} 252 (2012), no. 11, 6012–6060. 


\bibitem{BK05} R.F. Bass and M. Kassmann: Moritz Hölder continuity of harmonic functions with respect to operators of variable order. \emph{Comm. Partial Differential Equations} 30 (2005), no. 7-9, 1249–1259.

\bibitem{BK2005}R.F. Bass and M. Kassmann: Harnack inequalities for non-local operators of variable order. \emph{Trans. Amer. Math. Soc.} 357 (2005), no. 2, 837–850.

\bibitem{BL02} R.F. Bass and D.A. Levin: Harnack inequalities for jump processes. \emph{Potential Anal.} 17 (2002), no. 4, 375–388. 


\bibitem{BDVV}
S. Biagi, S. Dipierro, E. Valdinoci and E. Vecchi:
A Faber-Krahn inequality for mixed local and nonlocal operators. To appear in \emph{Journal d'Analyse Mathématique.} 

\bibitem{BDVV22}
S. Biagi, S. Dipierro, E. Valdinoci and E. Vecchi:
Mixed local and nonlocal elliptic operators: regularity and maximum principles, \emph{Communications in Partial Differential Equations} 47 (2022), no. 3, 585--629

\bibitem{BVDV} S. Biagi, E. Vecchi, S. Dipierro and E. Valdinoci:
Semilinear elliptic equations involving mixed local and nonlocal operators,
\emph{Proceedings of the Royal Society of Edinburgh Section A: Mathematics},
DOI:10.1017/prm.2020.75





\bibitem{BM20} A. Biswas and M. Modasiya: Regularity results of nonlinear perturbed stable-like operators. \emph{Differential Integral Equations} 33 (2020), no. 11-12, 597-624.

\bibitem{BM21} A. Biswas and M. Modasiya:
Mixed local-nonlocal operators: maximum principles, eigenvalue problems and their applications, preprint, JAMA (2025). https://doi.org/10.1007/s11854-025-0375-2

\bibitem{BMS22} A.Biswas, M. Modasiya and A. Sen:  Boundary regularity of mixed local-nonlocal operators and its application. \emph{Annali di Matematica (2022).} https://doi.org/10.1007/s10231-022-01256-0

\bibitem{BK22} A. Biswas and S. Khan: Existence-Uniqueness of nonlinear integro-differential equations with drift in $\Rd,$ \emph{SIAM J. Math. Anal.} 55 (2023), no. 5, 4378--4409.

\bibitem{B12} I.H. Biswas: On zero-sum stochastic differential games with jump-diffusion driven state: a viscosity solution framework. \emph{SIAM J. Control Optim.} 50 (2012), no. 4, 1823–1858.

\bibitem{BSW}B.Böttcher; R.L. Schilling and J.Wang: Lévy matters. III. Lévy-type processes: construction, approximation and sample path properties. With a short biography of Paul Lévy by Jean Jacod. Lecture Notes in Mathematics, 2099. Lévy Matters. Springer, Cham, 2013. xviii+199 pp. ISBN: 978-3-319-02683-1; 978-3-319-02684-8.

\bibitem{C12} L. A. Caffarelli: Non-local diffusions, drifts and games. Nonlinear partial differential equations, 37–52, Abel Symp., 7, Springer, Heidelberg, 2012


\bibitem{CCbook} L. A. Caffarelli and X. Cabré : Fully nonlinear elliptic equations. American Mathematical Society Colloquium Publications, 43. American Mathematical Society, Providence, RI, 1995. vi+104 pp. ISBN: 0-8218-0437-5 


\bibitem{CS09} L. A. Caffarelli and L. Silvestre:
 Regularity theory for fully nonlinear integro-differential equations,
\emph{Comm. Pure Appl. Math.} 62 (2009), 597--638.

\bibitem{CS11} L. A. Caffarelli and L. Silvestre:
Regularity results for nonlocal equations by approximation,
\emph{Arch. Ration.
Mech. Anal.} 200 (2011), no. 1, 59--88

\bibitem{CLD12}H. Chang Lara and G. Dávila: Regularity for solutions of nonlocal, nonsymmetric equations. \emph{Ann. Inst. H. Poincaré C Anal. Non Linéaire} 29 (2012), no. 6, 833–859.

\bibitem{RTbook} R. Cont and P. Tankov: Financial Modelling with Jump Processes. Chapman and Hall, 552 pages. ISBN 9781584884132. 




\bibitem{DM22} C. De Filippis and G. Mingione:
Gradient regularity in mixed local and nonlocal problems. \emph{Mathematische Annalen}. DOI: https://doi.org/10.1007/s00208-022-02512-7

\bibitem{DV21} S.Dipierro and E.Valdinoci: Description of an ecological niche for a mixed local/nonlocal dispersal: an evolution equation and a new Neumann condition arising from the superposition of Brownian and L\'{e}vy processes. \emph{Phys. A} 575 (2021), Paper No. 126052, 20 pp.

\bibitem{DPLV22} S. Dipierro; E.P. Lippi and E. Valdinoci : (Non) local logistic equations with Neumann conditions. \emph{Ann. Inst. H. Poincaré C Anal. Non Linéaire} 40 (2023), no. 5, 1093--1166.

\bibitem{DSV17}S. Dipierro, O. Savin, E. Valdinoci: All functions are locally s-harmonic up to a small error. J. Eur. Math. Soc. 19 (2017), no. 4, pp. 957–966









\bibitem{F09} M. Foondun:
Harmonic functions for a class of integro-differential operators, \emph{Potential Anal.} 31(1) (2009), 21--44

\bibitem{GM02}
M.G. Garroni and J.L. Menaldi: 
Second order elliptic integro-differential problems, 
Chapman \& Hall/CRC Research Notes in Mathematics, 430. Chapman \& Hall/CRC, Boca Raton, FL,
2002. xvi+221 pp.

\bibitem{GK22} P. Garain and J. Kinnunen:
On the regularity theory for mixed local and nonlocal quasilinear elliptic equations, to appear in \emph{Transactions of the AMS}, 2022

\bibitem{GL22} P. Garain and E. Lindgren:
Higher H\"{o}lder regularity for mixed local and nonlocal degenerate elliptic equations.\emph{ Calc. Var. Partial Differential Equations} 62 (2023), no. 2, Paper No. 67, 36 pp

\bibitem{EDG} E. DeGiorgi: Sulla differenziabilità e l'analiticità delle estremali degli integrali multipli regolari. (Italian) 
Mem. Accad. Sci. Torino. Cl. Sci. Fis. Mat. Nat. (3) 3 1957 25–43.

\bibitem{G14} G. Grubb: Local and nonlocal boundary conditions for $\mu$-transmission and fractional elliptic pseudodifferential operators. \emph{Anal. PDE} 7 (2014), no. 7, 1649–1682. 

\bibitem{G15} G. Grubb: Fractional Laplacians on domains, a development of Hörmander's theory of $\mu$-transmission pseudodifferential operators. \emph{Adv. Math.} 268 (2015), 478–528.

\bibitem{IMS} A. Iannizzotto, S. Mosconi and M. Squassina: Fine boundary regularity for the degenerate fractional p-Laplacian. \emph{J. Funct. Anal.} 279 (2020), no. 8, 108659, 54 pp.

\bibitem{K2007} M. Kassmann: The classical Harnack inequality fails for non-local operators, \emph{https://webdoc.sub.gwdg.de/ebook/serien/e/sfb611/360.pdf}

\bibitem{Kaz} J. L. Kazdan: Prescribing The Curvature Of A Riemannian Manifold, CBMS Reg. Conf. Ser. Math.57, Amer. Math. Soc., Providence, 1985.

\bibitem{KKLL} M. Kim, P. Kim, J. Lee and K-A Lee: Boundary regularity for nonlocal operators with kernel of variable orders. \emph{J. Funct. Anal.} 277 (2019), no. 1, 279--332.

\bibitem{KL20}M.Kim and K-A. Lee: Regularity for fully nonlinear integro-differential operators with kernels of variable orders. \emph{Nonlinear Anal.} 193 (2020), 111312, 27 pp.

\bibitem{KKL16} S.Kim, Y-C. Kim and K-A Lee: Regularity for fully nonlinear integro-differential operators with regularly varying kernels. \emph{Potential Anal.} 44 (2016), no. 4, 673–705.

\bibitem{KL12} Y-C. Kim and K-A. Lee: Regularity results for fully nonlinear integro-differential operators with nonsymmetric positive kernels. \emph{Manuscripta Math.} 139 (2012), no. 3-4, 291–319.

\bibitem{K22} S. Kitano : Harnck inequalities and Hölder estimates for fully nonlinear integro-differential equations with weak scaling conditions. \emph{J. Differential Equations} 376 (2023), 714–749.


\bibitem{K13} D.Kriventsov: $C^{1,\alpha}$ interior regularity for nonlinear nonlocal elliptic equations with rough kernels. \emph{Comm. Partial Differential Equations} 38 (2013), no. 12, 2081–2106.

\bibitem{K83} N. Krylov: Boundedly inhomogeneous elliptic and parabolic equations in a domain, \emph{Izv. Akad. Nauk SSSR Ser. Mat.} 47 (1983), 75--108.

\bibitem{Kr19} N.V. Krylov: All functions are locally s-harmonic up to a small error. \emph {Journal of Functional Analysis} 277 (2019), no.8, 2728--2733.

\bibitem{KS79} N.V. Krylov and M.V. Safonov: An estimate for the probability of a diffusion process hitting a set of positive measure. (Russian) \emph{Dokl. Akad. Nauk SSSR} 245 (1979), no. 1, 18–20. 

\bibitem{JM} J. Moser: A Harnack inequality for parabolic differential equations. \emph{Comm. Pure Appl. Math.} 17 (1964), 101–134.

\bibitem{M19} C. Mou: Existence of $C^\alpha$ solutions to integro-PDEs, 
\emph{Calc. Var. Partial Diff. Equ.} 58(4) (2019) , 1--28

\bibitem{MZ21} C. Mou and Y.P. Zhang:
Regularity Theory for Second Order Integro-PDEs,
\emph{Potential Anal} 54 (2021), 387--407

\bibitem{JN} J. Nash: Continuity of solutions of parabolic and elliptic equations. Amer. J. Math. 80 (1958), 931–954. 

\bibitem{RS14} X. Ros-Oton and J. Serra:
The Dirichlet problem for the fractional Laplacian: regularity up to the boundary,\emph{ J. Math. Pures Appl.} 101 (2014), 275--302.

\bibitem{RS16} X. Ros-Oton and J. Serra:
Boundary regularity for fully nonlinear integro-differential equations, \emph{Duke Math. J.} 165 (2016), 2079--2154.

\bibitem{RS17} X. Ros-Oton and J. Serra:
Boundary regularity estimates for nonlocal elliptic equations in $C^1$ and $C^{1,\alpha}$ domains, 
\emph{Annali di Matematica Pura ed Applicata} 196 (2017), 1637--1668.

\bibitem{S15} J. Serra: Regularity for fully nonlinear nonlocal parabolic equations with rough kernels. Calc. Var. Partial Differential Equations 54 (2015), no. 1, 615–629.

\bibitem{SSV} R. Schilling, R. Song, and Z. Vondra\v{c}ek: Bernstein Functions, Walter de Gruyter, 2010.


\bibitem{S2010} B. Sirakov: Solvability of uniformly elliptic fully nonlinear PDE. \emph{Arch. Ration. Mech. Anal.} 195 (2010), no. 2, 579–607.

\bibitem{S06} L.Silvestre: Hölder estimates for solutions of integro-differential equations like the fractional Laplace. \emph{Indiana Univ. Math. J.} 55 (2006), no. 3, 1155–1174. 

\end{thebibliography}
\end{document}